\documentclass[11pt,a4paper]{amsart}
\usepackage[T1]{fontenc}
\usepackage{lmodern}
\usepackage{amsmath,amsthm,amsfonts,amssymb,mathtools}
\usepackage{microtype}
\usepackage[textwidth=155mm,textheight=235mm,centering,headheight=14pt,headsep=7mm]{geometry}
\usepackage{cite,enumitem,needspace}
\usepackage{xurl}
\usepackage[unicode,hidelinks]{hyperref}
\numberwithin{equation}{section}
\setlist[enumerate,1]{label=\textup{(\roman*)},ref=\roman*,leftmargin=2.6em}
\newtheorem{theorem}{Theorem}[section]
\newtheorem{lemma}[theorem]{Lemma}
\newtheorem{proposition}[theorem]{Proposition}
\newtheorem{corollary}[theorem]{Corollary}
\theoremstyle{definition}
\newtheorem{definition}[theorem]{Definition}
\theoremstyle{remark}
\newtheorem{remark}[theorem]{Remark}
\theoremstyle{plain}
\newtheorem*{thmA}{Theorem A}

\newcommand{\C}{\mathbb C}
\newcommand{\Q}{\mathbb Q}

\newcommand{\OO}{\mathcal O}
\newcommand{\HH}{\mathcal H}
\newcommand{\GG}{\Gamma}
\newcommand{\Qab}{\mathbb Q^{\mathrm{ab}}}
\newcommand{\nrm}{\trianglelefteq}
\newcommand{\geb}{\mathrel{\ge_b}}
\DeclareMathOperator{\Irr}{Irr}

\DeclareMathOperator{\Bl}{Bl}
\DeclareMathOperator{\Aut}{Aut}
\DeclareMathOperator{\Gal}{Gal}
\DeclareMathOperator{\Ind}{Ind}

\DeclareMathOperator{\End}{End}
\DeclareMathOperator{\GL}{GL}

\DeclareMathOperator{\tr}{tr}
\DeclareMathOperator{\Br}{Br}
\DeclareMathOperator{\bl}{bl}
\DeclareMathOperator{\htc}{ht}
\DeclareMathOperator{\Syl}{Syl}

\hypersetup{
  pdftitle={A Galois-equivariant Dade--Glauberman--Nagao correspondence with central defect},
  pdfauthor={Shi Chen},
  pdfsubject={Research manuscript: ordinary characters and block H-triples},
  bookmarksnumbered=true
}

\title[Galois-equivariant DGN correspondence]{A Galois-equivariant
Dade--Glauberman--Nagao correspondence with central defect}
\author{Shi Chen}
\address{School of Mathematics and Statistics, Central China Normal University,
Wuhan 430079, China}
\email{chenshitjnu@163.com}
\keywords{Finite groups, ordinary characters, DGN correspondence,
Galois automorphisms, blocks, magic representations}
\date{}

\begin{document}
\begin{abstract}
We construct a correspondence between ordinary irreducible characters of
height zero above the generalized Dade--Glauberman--Nagao correspondence
for blocks with central defect. The correspondence is equivariant
under group and Galois automorphisms and satisfies a block relation
between $\mathcal H$-triples. For each pair of corresponding characters,
one pair of associated projective representations realizes matching
factor sets, equal central scalars, synchronized mixed comparison
functions and compatible block correspondences for every intermediate
group. The construction uses the multiplicity algebra over the group
algebra of the central defect subgroup, a normalized integral lift of
a residual semilinear realization, and a graded corner isomorphism.
It treats all central-character fibers, including those with nontrivial
central character. We also obtain a form suited to normal $p$-sections
in an Alperin--McKay--Navarro reduction.

\end{abstract}
\maketitle

\section*{Background and main theorem}
\phantomsection\addcontentsline{toc}{section}{Background and main theorem}

Let $p$ be a prime. Character correspondences above normal sections
are a key ingredient in reduction theorems for local-global
conjectures in the representation theory of finite groups. The
Alperin--McKay reduction of Sp\"ath \cite{Spa13} requires such
correspondences to be compatible with automorphisms and blocks of
intermediate groups. Galois refinements introduce an additional
compatibility requirement. Navarro's refinement of the McKay
conjecture \cite{Nav04} and the reduction theorem of Navarro,
Sp\"ath and Vallejo \cite{NSV20} provide the motivation and the
$\HH$-triple framework for the present work.

On the weight side, Feng, Fu and Zhou reduced the Navarro Alperin
weight conjecture to simple groups \cite{FFZ26a} and developed a
blockwise reduction in terms of $\HH$-triples \cite{FFZ26b}.
For $p$-solvable groups, Turull proved a strengthened
Alperin--McKay conjecture that also records character degree
residues, $p$-adic fields of values and local Schur indices
\cite{Tur13AM}.

The Dade--Glauberman--Nagao (DGN) correspondence is particularly
suited to these reduction arguments. Its origins lie in the
Glauberman correspondence for coprime action \cite{Gla68} and its
block-theoretic generalizations; see, for example, \cite{Dad80}.
At the level of block algebras, Koshitani and Michler \cite{KM01}
established Morita equivalences between the Brauer correspondents
of suitable blocks related by the Glauberman correspondence.
Navarro and Sp\"ath extended the DGN correspondence to characters
in blocks with normal defect groups \cite{NS14DGN}. In the
central-defect setting, they also constructed equivariant
correspondences of height-zero characters satisfying block
character-triple relations \cite[Theorem~5.13]{NS14}.

The algebraic approach relates character theory to the local
structure of $G$-algebras, as developed by Brou\'e and Puig
\cite{BP80}. Dade's theory of block extensions \cite{Dad73}
provides a Clifford-theoretic framework for blocks over normal
subgroups. Related structural results concern the source algebras
of nilpotent blocks \cite{Pui88}, their extensions \cite{KP90,PZ12},
and Morita equivalences associated with the Glauberman correspondence
\cite{GM15}.

Several constructions describe the additional properties of the
resulting character correspondences. Turull studied correspondences above
Glauberman correspondents, including compatibility with fields of
values and local Schur indices \cite{Tur08}. Ladisch's theory of
magic representations constructs character correspondences from
algebra isomorphisms and describes their behavior over different
coefficient fields \cite{Lad11}. Marcus and Minu\c{t}\u{a}
studied graded endomorphism algebras and graded Morita equivalences,
including a graded analogue of the butterfly theorem \cite{MM20},
and developed their connection with central relations between
character triples \cite{MM21}.

Turull's Brauer--Clifford group provides an algebraic framework
for character correspondences that preserve fields of values and
Schur indices \cite{Tur09a}. Its subgroup represented by full
matrix algebras has a description in terms of second cohomology
\cite{Tur09b}. Ladisch studied the Schur--Clifford subgroup
arising from finite-group character theory \cite{Lad15SC} and
reductions of character triples that preserve their Clifford
theory over the field of values, together with the Galois action
\cite{Lad16Gal}. A further construction of module and character
correspondences uses Turull's endoisomorphisms \cite{Tur13Endo}.

For irreducible Brauer characters, Fu constructed correspondences
above the DGN correspondence that are equivariant under both group
and Galois automorphisms, using modular $\HH$-triples and a fusion
homomorphism \cite{Fu26}. Here we construct an integral
correspondence for ordinary characters in the central-defect setting.
The construction treats all central-character fibers simultaneously,
including those on which the central defect subgroup acts
nontrivially, and yields a block relation between $\HH$-triples.
For each pair of corresponding characters, one pair of associated
projective representations realizes the common factor set, the
central scalar identities, the mixed comparison functions and the
block correspondences for every intermediate group.

We write $\HH=\HH_p\leq\Gal(\Qab/\Q)$ for the subgroup consisting
of the automorphisms whose action on all roots of unity of order
prime to $p$ is a common integral power of the map
$\zeta\mapsto\zeta^p$. Our main result is the following.
The notation and the definition of the relation $\ge_b$ are given
in Section~\ref{sec:prelim}.

\Needspace{12\baselineskip}
\begin{thmA}
\label{dgn:main}
Let $N\leq M$ be normal subgroups of a finite group $A$ such that
$M/N$ is a $p$-group. Let $b\in\Bl(N)$ have defect group
$Z\leq Z(M)$, and let $\theta\in\Irr(b)$ satisfy
\[
 \theta^m=\theta\qquad\text{for every }m\in M.
\]
Set
\[
 \Omega_\theta=\{\theta^\gamma\mid\gamma\in\HH\},
\]
and assume that $\theta^a\in\Omega_\theta$ for every $a\in A$.
Let $B$ be the unique block of $M$ covering $b$, and let $D$ be a
defect group of $B$. Set
\[
 H=N_A(D),\qquad L=N_N(D),\qquad M'=N_M(D),
\]
and let $\phi\in\Irr(L)$ be the generalized DGN correspondent
of $\theta$. Set $\Omega_\phi=\{\phi^\gamma\mid\gamma\in\HH\}$.
Then there is an $H\times\HH$-equivariant bijection
\[
 \Delta_D:\Irr_0(M\mid\Omega_\theta)
       \longrightarrow\Irr_0(M'\mid\Omega_\phi)
\]
that maps the fiber above $\theta^\gamma$ to the fiber above
$\phi^\gamma$ for every $\gamma\in\HH$, and sends each character
to a character in the Brauer correspondent of its block.
Moreover, if $\xi\in\Irr_0(M\mid\Omega_\theta)$ and
$\eta=\Delta_D(\xi)$, then
\[
 (A_{\xi^{\HH}},M,\xi)_{\HH}\ge_b
 (H_{\eta^{\HH}},M',\eta)_{\HH}.
\]
For each such pair $(\xi,\eta)$, the same pair of associated
projective representations satisfies all the conditions defining
this relation.
\end{thmA}

Here, for a set $\Omega$ of irreducible characters of a normal
subgroup of $U$, the notation $\Irr_0(U\mid\Omega)$ denotes the
height-zero irreducible characters of $U$ lying over some member
of $\Omega$. The group $G_{\psi^{\HH}}$ is the stabilizer in $G$
of the Galois orbit $\{\psi^\gamma\mid\gamma\in\HH\}$. For every $\gamma\in\HH$, the block
$B^\gamma$ has $D$ as a defect group, so its Brauer correspondent
is taken in $M'=N_M(D)$. The theorem allows nontrivial central
characters on $Z$ and assumes only that $Z\leq Z(M)$; it does not
require $Z\leq Z(A)$.

The proof begins with the multiplicity algebra over $R=\OO Z$,
where $\OO$ is the valuation ring of a sufficiently large splitting
$p$-modular system. This algebra is a matrix algebra over $R$, and
its reduction modulo $J(R)$ is a Dade algebra. A fusion homomorphism
provides a crossed realization of the action of the mixed stabilizer
on this reduction. A determinant normalization then yields an
integral lift, from which we construct a graded corner
correspondence on all central-character fibers.

The remaining steps establish the character-triple and block
conditions for this correspondence. Compressing associated
projective representations through the corner gives matching mixed
comparison functions. We then identify the resulting tensor
correspondence with the corner correspondence used in the block
argument. This ensures that the block conditions are satisfied by
the same pair of projective representations. Block compatibility
is proved using central test groups and a normalized trace criterion.

Section~\ref{sec:prelim} introduces the notation and recalls the
character-triple tools. Section~\ref{sec:central} constructs the
integral multiplicity algebra and describes its central-character
fibers. Section~\ref{sec:lift} gives the semilinear realization
and its normalized integral lift, and Section~\ref{sec:corner}
constructs the graded corner correspondence and establishes its
degree properties. Sections~\ref{sec:mixed} and~\ref{sec:blocks}
prove compatibility of the mixed comparison functions and blocks,
respectively. Finally, Section~\ref{sec:completion} completes the
proof of Theorem~A and derives a form suited to the local step of
an Alperin--McKay--Navarro reduction.

\section{Preliminaries}\label{sec:prelim}

We fix the notation for coefficient fields, character fibers and
$\HH$-triples, and recall the correspondences and block criteria
used below. We also introduce the fusion homomorphisms and corner
algebras used in the subsequent constructions.

\subsection{Coefficient fields and Galois actions}

Throughout this paper, all groups are finite and all characters are
ordinary unless otherwise stated. We fix a prime $p$ and denote by
$\HH=\HH_p$ the subgroup of $\Gal(\Qab/\Q)$ consisting of the
automorphisms whose action on all roots of unity of order prime to $p$
is a common integral power of the map $\xi\mapsto\xi^p$.
Our conventions for $\HH$-triples follow \cite[Section~1]{NSV20}.
For background on ordinary character theory and blocks, we refer to
\cite{Isa06,Nav98}; for valuations and finite extensions of $\Q_p$,
see \cite{Ser79}.
Navarro's monograph \cite{Nav18} gives a systematic account of group
and Galois actions on characters and their role in the McKay conjecture.

Fix an embedding $\Qab\hookrightarrow\overline{\Q}_p$. For each
construction, choose a sufficiently large finite cyclotomic field
containing the relevant character values and projective matrix
entries. Let $K$ be its completion at the chosen prime, and let
$(K,\OO,k)$ be the resulting splitting $p$-modular system. We write
$J(\Lambda)$ for the Jacobson radical of a ring $\Lambda$. Thus
$\OO$ is the valuation ring of $K$, with maximal ideal $J(\OO)$
and residue field $k=\OO/J(\OO)$. For a finite group $X$,
we write $K[X]$ for its group algebra over the coefficient field $K$.

The action of $\HH$ preserves the chosen prime and induces compatible
actions on $K$, $\OO$ and $k$. Each construction uses only the finite
image of $\HH$ on these coefficients. For $u\in\OO$, we write $u^*$
for its image in $k$. When the coefficient system is enlarged, the
chosen matrices and algebra maps are extended by scalars.

We use right exponential notation for group and Galois actions.
Thus, if $X\nrm A$, $h\in A$,
$\gamma\in\HH$ and $\chi\in\Irr(X)$, we set
\[
 \chi^{h\gamma}(x)=\gamma\bigl(\chi(hxh^{-1})\bigr)
 \qquad(x\in X).
\]
The group and Galois actions commute. We denote the inertia group
of $\chi$ by $A_\chi$ and the stabilizer of its Galois orbit by
\[
 A_{\chi^{\HH}}=\{h\in A\mid\chi^h\in\chi^{\HH}\}.
\]
For $a=(h,\gamma)$, we also write $\chi^a=\chi^{h\gamma}$.

For a $p$-block $b$ of $X$, let $\Irr_0(b)$ denote the set of its
irreducible characters of height zero. If $D$ is a defect group of $b$,
the height of $\chi\in\Irr(b)$ is determined by
\begin{equation}\label{pre:eq:height}
 \chi(1)_p=|X:D|_p\,p^{\htc(\chi)}\qquad(\chi\in\Irr(b)).
\end{equation}
Here $r_p$ is the $p$-part of a positive integer $r$. We write
$\Irr_0(X)=\bigcup_{b\in\Bl(X)}\Irr_0(b)$ and, for a specified
$p$-subgroup $D\leq X$, set
\[
 \Irr_0(X\mid D)=
 \bigcup_{\substack{b\in\Bl(X)\\D\text{ a defect group of }b}}
 \Irr_0(b).
\]
Here $D$ is fixed as a subgroup of $X$.

\subsection{Character fibers and group actions}

Let $N\nrm X$ and $\theta\in\Irr(N)$. We use the standard notation
\[
 \Irr(X\mid\theta)=
 \{\chi\in\Irr(X)\mid\langle\chi_N,\theta\rangle_N>0\}.
\]
If $\Theta\subseteq\Irr(N)$, then
$\Irr(X\mid\Theta)=\bigcup_{\theta\in\Theta}\Irr(X\mid\theta)$.
The subscript $0$ restricts these sets to height-zero characters.
We use the analogous notation within a fixed block. If a defect
group $D$ is specified, we also intersect with $\Irr_0(X\mid D)$.

By Clifford theory, the irreducible constituents of $\chi_N$ form
a single $X$-orbit for each $\chi\in\Irr(X)$. In particular,
\[
 \Irr(X\mid\theta)=\Irr(X\mid\theta')
 \quad\text{if }\theta'=\theta^x\text{ for some }x\in X,
\]
while distinct $X$-orbits give disjoint sets. The natural fiber map
for a normal subgroup therefore takes values in $\Irr(N)/X$.

\begin{definition}\label{pre:def:central-fiber}
Let $Z\leq Z(X)$ and $\lambda\in\Irr(Z)$. The central-character
fiber above $\lambda$ is
\[
 \Irr(X\mid\lambda)=
 \{\chi\in\Irr(X)\mid\chi_Z=\chi(1)\lambda\}.
\]
For $b\in\Bl(X)$, set
$\Irr_0(b\mid\lambda)=\Irr_0(b)\cap\Irr(X\mid\lambda)$.
\end{definition}

By Schur's lemma, each $\chi\in\Irr(X)$ determines a unique
$\lambda_\chi\in\Irr(Z)$ with $\chi_Z=\chi(1)\lambda_\chi$.
Definition~\ref{pre:def:central-fiber} describes the fibers of
$\chi\mapsto\lambda_\chi$. In particular,
\begin{equation}\label{pre:eq:central-partition}
 \Irr_0(b)=\bigsqcup_{\lambda\in\Irr(Z)}\Irr_0(b\mid\lambda).
\end{equation}
Some of these fibers can be empty.

Suppose that $X,N\nrm A$. Restriction commutes with both actions,
so for $a=(h,\gamma)\in A\times\HH$ we have
\[
 \Irr(X\mid\theta)^a=\Irr(X\mid\theta^a).
\]
If $Z\leq Z(X)$ is $A$-stable, central-character fibers transform
in the same way. For height-zero characters in a fixed block,
\begin{equation}\label{pre:eq:fiber-transport}
 \Irr_0(b\mid\lambda)^a=
 \Irr_0(b^a\mid\lambda^a).
\end{equation}
The Galois action also preserves specified defect groups.

\begin{lemma}\label{pre:lem:galois-defects}
Let $b\in\Bl(X)$ have defect group $D$, and let $\gamma\in\HH$.
Then $D$ is a defect group of $b^\gamma$. Brauer correspondence
commutes with this Galois action.
\end{lemma}
\begin{proof}
We first describe the action on the coefficient algebras. By the
choice of the splitting $p$-modular system $(K,\OO,k)$, the
automorphism $\gamma$ preserves $\OO$. It therefore preserves its
unique maximal ideal $J(\OO)$ and induces a map
\[
 \bar\gamma:k\longrightarrow k,\qquad
 u+J(\OO)\longmapsto\gamma(u)+J(\OO).
\]
This map is well defined: if $u-v\in J(\OO)$, then
$\gamma(u)-\gamma(v)=\gamma(u-v)\in J(\OO)$.
It preserves addition, multiplication and the identity, and the map
induced by $\gamma^{-1}$ is its inverse. Thus $\bar\gamma$ is a
field automorphism of $k$. In the notation introduced above,
\[
 \bar\gamma(u^*)=\gamma(u)^*\qquad(u\in\OO).
\]

For a subgroup $Y\leq X$, define the coefficient action on $kY$ by
\[
 \left(\sum_{y\in Y}a_y y\right)^\gamma
   =\sum_{y\in Y}\bar\gamma(a_y)y.
\]
This is a ring automorphism fixing every element of $Y$ and
satisfying $(au)^\gamma=\bar\gamma(a)u^\gamma$ for $a\in k$
and $u\in kY$. Hence it is $\bar\gamma$-semilinear. The analogous
coefficient actions on $\OO Y$ and $K[Y]$ use $\gamma$ itself and
commute with reduction modulo $J(\OO)$. These automorphisms preserve
centers and permute primitive central idempotents: a nontrivial
decomposition of the image of such an idempotent would give one
of the original idempotent after applying the inverse automorphism.

We next identify this permutation with the Galois action on blocks.
For $\chi\in\Irr(X)$, the corresponding primitive central
idempotent of $K[X]$ is
\[
 e_\chi=\frac{\chi(1)}{|X|}
       \sum_{x\in X}\chi(x^{-1})x.
\]
Since $\gamma$ fixes rational numbers, applying it coefficientwise
gives $e_\chi^\gamma=e_{\chi^\gamma}$. Let $E_b\in Z(\OO X)$
be the block idempotent of $b$, and let $e_b=E_b^*\in Z(kX)$ be
its reduction. In $K[X]$ we have
\[
 E_b=\sum_{\chi\in\Irr(b)}e_\chi,
 \qquad
 E_b^\gamma=\sum_{\chi\in\Irr(b)}e_{\chi^\gamma}.
\]
As $E_b^\gamma$ is again a primitive central idempotent of $\OO X$,
the latter equality shows that its irreducible characters are
precisely the elements of $\Irr(b)^\gamma$. It is therefore the
idempotent of $b^\gamma$. Reducing this identity yields
\[
 e_{b^\gamma}=e_b^\gamma.
\]

Let $Q\leq X$ be a $p$-subgroup. The coefficient action commutes
with conjugation by $Q$, so it preserves $(kX)^Q$. The Brauer
homomorphism is the map
\[
 \Br_Q^X:(kX)^Q\longrightarrow kC_X(Q),\qquad
 \sum_{x\in X}a_xx\longmapsto
 \sum_{x\in C_X(Q)}a_xx.
\]
Consequently, for $u=\sum_{x\in X}a_xx\in(kX)^Q$, we obtain
\[
 \Br_Q^X(u^\gamma)
 =\sum_{x\in C_X(Q)}\bar\gamma(a_x)x
 =\Br_Q^X(u)^\gamma.
\]
Since the coefficient action on $kC_X(Q)$ is bijective, this gives
\[
 \Br_Q^X(e_{b^\gamma})\neq0
 \quad\Longleftrightarrow\quad
 \Br_Q^X(e_b)\neq0.
\]
By the Brauer-homomorphism characterization of defect groups, a
defect group of a block is a $p$-subgroup maximal under inclusion
among those on which its block idempotent has nonzero Brauer image;
see \cite[Sections~6.1--6.2]{Lin18II}.
In particular, $\Br_D^X(e_{b^\gamma})\neq0$. If a $p$-subgroup
$Q$ properly containing $D$ also had
$\Br_Q^X(e_{b^\gamma})\neq0$, the displayed equivalence would
give $\Br_Q^X(e_b)\neq0$, contrary to the maximality of $D$ for
$b$. Thus $D$ is a defect group of $b^\gamma$. The same argument
with $\gamma^{-1}$ shows that $b$ and $b^\gamma$ have exactly the
same defect groups as subgroups of $X$.

Finally, let $c$ be the Brauer correspondent of $b$ in $N_X(D)$.
In its idempotent formulation, Brauer's first main theorem states
that
\[
 e_c=\Br_D^X(e_b)
 \quad\text{in }kN_X(D);
\]
see \cite[Section~6.7]{Lin18II}. Here the Brauer image lies in
$kC_X(D)\subseteq kN_X(D)$, and its primitivity is understood in
the center of the normalizer algebra $kN_X(D)$; it may decompose
as a sum of block idempotents in $kC_X(D)$.
Applying the coefficient action and using the identities above,
we find
\[
 e_{c^\gamma}
 =e_c^\gamma
 =\Br_D^X(e_b)^\gamma
 =\Br_D^X(e_{b^\gamma}).
\]
Since $D$ is a defect group of $b^\gamma$, Brauer's first main
theorem identifies the last expression with the idempotent of its
Brauer correspondent in $N_X(D)$. Therefore this correspondent is
$c^\gamma$, as required.
\end{proof}

\subsection{Projective representations and comparison functions}

We use the character-triple conventions of \cite[Chapter~11]{Isa06}
and \cite[Section~1]{NSV20}. For a projective representation
$\mathcal P$ with factor set $\alpha$, our convention is
\[
 \mathcal P(x)\mathcal P(y)=\alpha(x,y)\mathcal P(xy),\qquad \mathcal P(1)=1.
\]
Let $N\nrm G$ and $\theta\in\Irr(N)$. A projective representation
$\mathcal P$ of $G_\theta$ associated with $\theta$ restricts to a
representation affording $\theta$ and satisfies
$\mathcal P(nx)=\mathcal P(n)\mathcal P(x)$ and $\mathcal P(xn)=\mathcal P(x)\mathcal P(n)$ for $n\in N$.
Its factor set is inflated from $G_\theta/N$. Throughout, the matrix
entries lie in $\Qab$ and the factor-set values are roots of unity.

For $a=(h,\gamma)\in(G\times\HH)_\theta$, the element $h$
normalizes $G_\theta$, and the projective representation
\[
 \mathcal P^a(x)=\mathcal P(hxh^{-1})^\gamma
\]
is again associated with $\theta$. There is a unique function
$\mu_a:G_\theta\to(\Qab)^{\times}$, constant on $N$-cosets and
with $\mu_a(1)=1$, such that
\begin{equation}\label{pre:eq:comparison}
 \mathcal P^a(x)=\mu_a(x)T_a\mathcal P(x)T_a^{-1}\qquad(x\in G_\theta)
\end{equation}
for some invertible matrix $T_a$ independent of $x$; see
\cite[Lemma~1.4]{NSV20}. We call $\mu_a$ the comparison function
of $\mathcal P$ for $a$. In particular, $\mu_a|_N=1$.

For a function $f:G_\theta\to(\Qab)^\times$, define
$df(x,y)=f(x)f(y)f(xy)^{-1}$. Comparing products in
\eqref{pre:eq:comparison} gives
\begin{equation}\label{pre:eq:comparison-coboundary}
 d\mu_a=\alpha^a\alpha^{-1}.
\end{equation}
If $f$ is constant on $N$-cosets and normalized by $f(1)=1$,
replacing $\mathcal P$ by $f\mathcal P$ gives
\begin{equation}\label{pre:eq:twist}
 \widetilde\alpha=(df)\alpha,
 \qquad \widetilde\mu_a=\mu_a\frac{f^a}{f}.
\end{equation}
Thus equality of factor sets does not ensure equality of comparison
functions. A twist by a linear character $f$ leaves $\alpha$
unchanged but can alter $\mu_a$ through the factor $f^a/f$.

\subsection{Block relations between \texorpdfstring{$\HH$}{H}-triples}

Let $N\nrm G$ and $\theta\in\Irr(N)$. We call
$(G,N,\theta)_\HH$ an $\HH$-triple if
$\theta^g\in\theta^\HH$ for every $g\in G$.
We combine the $\HH$-triple relation of
\cite[Definition~1.5]{NSV20} with the block conditions of
\cite[Definition~3.6]{NS14}, requiring all conditions to be realized
by the same pair of associated projective representations.

\Needspace{10\baselineskip}
\begin{definition}\label{def:full}\label{cliff:def:full-block-relation}
Let $(G,N,\theta)_\HH$ and $(H,M,\phi)_\HH$ be $\HH$-triples,
where $H\leq G$ and $M=N\cap H$. We write
\[
 (G,N,\theta)_\HH\geb(H,M,\phi)_\HH
\]
if the following conditions hold.
\begin{enumerate}
\item $G=NH$, $C_G(N)\leq H$, and
$(H\times\HH)_\theta=(H\times\HH)_\phi$.
\item There are associated projective representations
\[
 \mathcal P:G_\theta\longrightarrow\GL_{\theta(1)}(\Qab),\qquad
 \mathcal P':H_\phi\longrightarrow\GL_{\phi(1)}(\Qab)
\]
whose factor sets take values in roots of unity and agree on
$H_\phi\times H_\phi$. For every $c\in C_G(N)$, the matrices
$\mathcal P(c)$ and $\mathcal P'(c)$ are scalar matrices with the
same scalar.
\item For every $a\in(H\times\HH)_\theta$, the comparison functions
of this pair satisfy $\mu'_a=\mu_a|_{H_\phi}$.
\item The blocks $\bl(\theta)$ and $\bl(\phi)$ have a common
defect group $D$ with $N_N(D)\leq M$. For every $N\leq W\leq G_\theta$,
the correspondence
\[
 \tau_W:\Irr(W\mid\theta)\longrightarrow\Irr(W\cap H\mid\phi)
\]
determined by this same pair $\mathcal P,\mathcal P'$ satisfies
\[
 \bl\bigl(\tau_W(\chi)\bigr)^W=\bl(\chi)
 \qquad(\chi\in\Irr(W\mid\theta)).
\]
The block induction in this equality is required to be defined.
\end{enumerate}
\end{definition}

Omitting condition (iv) defines $\ge_c$. By condition (i),
$H_\theta=H_\phi$ and $G_\theta=NH_\phi$. Hence
$W/N\cong(W\cap H)/M$ whenever $N\leq W\leq G_\theta$.
The next subsection recalls the correspondences in condition (iv).

Both relations are unchanged under the simultaneous replacement
of $\mathcal P,\mathcal P'$ by $f\mathcal P,f|_{H_\phi}\mathcal P'$, where
$f$ is normalized, constant on $N$-cosets and root-of-unity valued.
Equation~\eqref{pre:eq:twist} gives compatible changes to the factor
sets and comparison functions, while the inverse twist on the
quotient representation leaves the induced character correspondences
unchanged.

\subsection{Clifford correspondences and block criteria}

Let $(\mathcal P,\mathcal P')$ realize
$(G,N,\theta)_\HH\ge_c(H,M,\phi)_\HH$, and let
$N\leq W\leq G_\theta$. If $\mathcal Q$ is an irreducible projective
representation of $W/N$ with factor set inverse to that of $\mathcal P_W$,
then the induced correspondence is given by
\begin{equation}\label{pre:eq:tensor-correspondence}
 \tau_W\bigl(\tr(\mathcal Q\otimes \mathcal P_W)\bigr)=
 \tr\bigl(\mathcal Q\otimes \mathcal P'_{W\cap H}\bigr).
\end{equation}
Here $\mathcal Q$ is transported through the natural isomorphism
$W/N\cong(W\cap H)/M$. Clifford theory parametrizes both character
sets by the same irreducible projective representations. Consequently,
$\tau_W$ is a bijection satisfying
\begin{equation}\label{pre:eq:relative-degree}
 \frac{\tau_W(\chi)(1)}{\phi(1)}=
 \frac{\chi(1)}{\theta(1)}.
\end{equation}
For $a=(h,\gamma)\in(H\times\HH)_\theta$, equality of the
comparison functions yields
\begin{equation}\label{pre:eq:tensor-covariance}
 \tau_{W^h}(\chi^{h\gamma})=
 \tau_W(\chi)^{h\gamma},\qquad W^h=h^{-1}Wh.
\end{equation}
Indeed, transport replaces the quotient factor by
$\mathcal Q^a\mu_a$ on the source and by $\mathcal Q^a\mu'_a$ on
the local side. These agree under the quotient identification;
see also \cite[Lemma~1.9(a)]{NSV20}.

For an arbitrary normal subgroup $N\nrm W$ and $\theta\in\Irr(N)$,
the usual Clifford correspondence is the induction bijection
\[
 \Irr(W_\theta\mid\theta)\longrightarrow\Irr(W\mid\theta),
 \qquad\eta\longmapsto\eta^W.
\]
The tensor correspondence \eqref{pre:eq:tensor-correspondence}
applies when the bottom character is invariant in the group under
consideration. The block condition in Definition~\ref{def:full}
refers to this specified tensor correspondence.

For block membership, we use the central-character criterion
\cite{Nav98}. If $C$ is the $X$-conjugacy class of $x$ and
$\chi\in\Irr(X)$, the central character $\omega_\chi$ satisfies
\[
 \omega_\chi(C^+)=\frac{|C|\chi(x)}{\chi(1)}\in\OO,
 \qquad C^+=\sum_{y\in C}y.
\]
Two irreducible characters belong to the same block precisely when
these values agree modulo $J(\OO)$ for every class sum, or,
equivalently, when their central characters induce the same map
from $Z(\OO X)$ to $k$. The criterion follows because the center
of each block of $kX$ is local, while central block idempotents
distinguish different blocks.

For the associated projective representations, we use the trace
criteria of \cite[Lemma~4.2 and Theorem~4.4]{NS14}. Their hypotheses
on the local group, integrality and defect groups are verified for
the chosen representations at each application.

To deduce preservation of height zero from
\eqref{pre:eq:relative-degree}, one must also identify the defect
groups of the relevant blocks. Indeed, if $\chi$ and
$\chi'=\tau_W(\chi)$ belong to blocks with a common defect group
$Q\leq W\cap H$, then \eqref{pre:eq:height} gives
\[
 p^{\htc(\chi)-\htc(\chi')}=
 \frac{\chi(1)_p}{\chi'(1)_p}\,|W:W\cap H|_p^{-1}.
\]
\subsection{Dade algebras and fusion homomorphisms}
\label{pre:sec:fusion}

For general background on $G$-algebras and their use in modular
representation theory, we refer to Th\'evenaz \cite{The95}.

Recall that $k$ is a finite field of characteristic $p$, and let
$P$ be a finite $p$-group. A \emph{Dade $P$-algebra} over $k$
is a full matrix $k$-algebra $S$ with a right $P$-action by
$k$-algebra automorphisms and a $P$-stable $k$-basis containing a
$P$-fixed element. For $Q\leq P$, let $S^Q$ denote the fixed-point
subalgebra and define
\[
 \operatorname{Tr}_Q^P(x)=\sum_{u\in Q\backslash P}x^u
 \qquad(x\in S^Q),
\]
where the sum uses coset representatives. The \emph{Brauer quotient}
and the canonical Brauer homomorphism are
\[
 S(P)=S^P\Big/\sum_{Q<P}\operatorname{Tr}_Q^P(S^Q),
 \qquad \Br_P:S^P\longrightarrow S(P).
\]
There is a unique homomorphism $\rho:P\to S^\times$ such that
$x^u=\rho(u)^{-1}x\rho(u)$. Moreover, $S(P)$ is a full matrix
$k$-algebra; see \cite[Section~4.1 and Lemma~4.1]{Fu26}.
Fix identifications $S=\End_k(V)$ and $S(P)=\End_k(W)$.
Restriction of scalars embeds these algebras into
$\End_{\mathbb F_p}(V)$ and $\End_{\mathbb F_p}(W)$, respectively.

\begin{definition}\label{pre:def:fusion-homomorphism}
Relative to these identifications, a \emph{fusion homomorphism}
for $S$ is a group homomorphism
\[
 \Phi:N_{\GL_{\mathbb F_p}(V)}(\rho(P))
       \longrightarrow\GL_{\mathbb F_p}(W)
\]
such that
\[
 \Phi(c)=\Br_P(c)\qquad\bigl(c\in(S^P)^\times\bigr),
\]
where the natural algebra embeddings are understood.
\end{definition}

Existence follows from \cite[Theorem~4.2]{Fu26}; the definition
does not impose uniqueness.
When the scalar map identifies $S(P)$ with $k$, we take $W=k$.
Writing $m_z^V$ and $m_z^k$ for scalar multiplication by $z$ on
$V$ and $k$, respectively, the defining identity gives
\[
 \Phi(m_z^V)=m_z^k\qquad(z\in k^\times).
\]
Applied to the residual multiplicity algebra in Section~\ref{sec:lift}, this
identity supplies the scalar normalization of the semilinear
implementing operators.

\subsection{Corner algebras}
\label{pre:sec:corners}

\begin{definition}\label{pre:def:corner-algebra}
Let $R$ be a commutative ring, let $A$ be a unital $R$-algebra,
and let $e\in A$ be an idempotent. The \emph{corner algebra}
associated with $e$ is
\[
 eAe=\{eae\mid a\in A\},
\]
with addition and multiplication inherited from $A$, identity element
$e$, and $R$-algebra structure given by $r\mapsto re$.
\end{definition}

For $u\in(eAe)^\times$, its inverse in the corner algebra satisfies
$uu^{-1}=u^{-1}u=e$. We use this convention whenever inverses are
taken in a corner algebra.

\section{Blocks with central defect}\label{sec:central}\label{app:dgn}

We keep the hypotheses and notation of Theorem~A. Thus
$N\leq M$ are normal subgroups of $A$, the quotient $M/N$ is a
$p$-group, and $b\in\Bl(N)$ has defect group $Z\leq Z(M)$.
The character $\theta\in\Irr(b)$ is $M$-invariant and satisfies
$\theta^a\in\theta^\HH$ for every $a\in A$.
Let $B$ be the unique block of
$M$ covering $b$, and let $D$ be a defect group of $B$. Recall that
$H=N_A(D)$, $L=N_N(D)$ and $M'=N_M(D)$, and that $\phi$
is the generalized DGN correspondent of $\theta$ in $L$.

We work over a sufficiently large splitting $p$-modular system
$(K,\OO,k)$ as in Section~\ref{sec:prelim}, with $\OO$ complete
and $k=\OO/J(\OO)$.

\subsection{The central-defect configuration}

The generalized DGN configuration of \cite[Section~5]{NS14} gives
\begin{equation}\label{dgn:eq:basic-groups}
 M=ND,\quad N\cap D=Z,\quad M'=LD,\quad
 L/Z=C_{N/Z}(D/Z).
\end{equation}
These identities describe the global and local groups in terms of
the same defect group $D$.

\subsection{Defect groups and stabilizer factorizations}

The Galois condition on $\theta$ yields the following Frattini-type
factorizations.

\begin{lemma}\label{dgn:lem:central91-frattini}
For every $\gamma\in\HH$, the group $D$ is a defect group of $B^\gamma$.
Moreover,
\[
 A=MH=NH,\qquad A_\theta=NH_\theta.
\]
\end{lemma}

\begin{proof}
Let $e\in Z(\OO M)$ be the block idempotent corresponding to $B$,
and fix $\gamma\in\HH$. The coefficient action of $\gamma$
sends $e$ to the block idempotent of $B^\gamma$ and commutes
with every relative trace map
$\operatorname{Tr}_P^M:(\OO M)^P\to(\OO M)^M$, where $P\leq M$.
Indeed,
\[
 \operatorname{Tr}_P^M(u^\gamma)
   =\operatorname{Tr}_P^M(u)^\gamma
 \qquad\bigl(u\in(\OO M)^P\bigr),
\]
since $\gamma$ acts on coefficients and fixes the group elements.
As $\gamma$ is invertible, $e$ belongs to the image of
$\operatorname{Tr}_P^M$ if and only if $e^\gamma$ does.
Defect groups are the minimal subgroups with this property
\cite[Definition~6.1.1]{Lin18II}. Thus $B$ and $B^\gamma$ have
the same defect groups, proving the first assertion.

Let $a\in A$. Choose $\gamma\in\HH$ such that
$\theta^a=\theta^\gamma$. Then $b^a=b^\gamma$, and uniqueness
of the covering block gives $B^a=B^\gamma$. Hence $D^a$ and
$D$ are defect groups of the same block. By
\cite[Theorem~6.1.2(iii)]{Lin18II}, there is $m\in M$ such
that $D^{am}=D$. Thus $am\in H$ and $A=HM=MH$, since
$M\nrm A$. Now $M=ND$ and $D\leq H$ give $A=NH$.

Finally, $N\leq A_\theta$. If $xh\in A_\theta$, with
$x\in N$ and $h\in H$, then $\theta^h=\theta^{xh}=\theta$,
so $h\in H_\theta$. This proves $A_\theta=NH_\theta$.
No $A$-stability of $B$ is needed.
\end{proof}

\subsection{Matrix algebras over the central group algebra}

Let $b'\in\Bl(L)$ be the block covered by the Brauer correspondent
of $B$, and write $q$ for passage to the central quotient by $Z$.
The blocks $q(b)$ and $q(b')$ have defect zero. Denote their unique
ordinary irreducible characters by $\theta_0$ and $\phi_0$,
respectively. They correspond under the $D/Z$-DGN correspondence.
Set
\[
 d=\theta_0(1),\qquad e=\phi_0(1),\qquad
 n=[(\theta_0)_{L/Z},\phi_0].
\]

\begin{lemma}\label{dgn:lem:matrix-base}
The ring $R=\OO Z$ is a complete commutative local ring, with
maximal ideal
\[
 J(R)=J(\OO)R+\langle z-1:z\in Z\rangle.
\]
There are isomorphisms of $R$-algebras
\[
 \OO Nb\cong M_d(R),\qquad \OO Lb'\cong M_e(R),
\]
where $R$ acts through its central embeddings in the block algebras.
\end{lemma}

\begin{proof}
Since $Z$ is an abelian $p$-group, the algebra $kZ$ is local
and its augmentation ideal is nilpotent. It follows that $R$
is local with the displayed maximal ideal, and that the
$J(R)$-adic and $J(\OO)R$-adic topologies agree. As an
$\OO$-module, $R$ has the elements of $Z$ as a basis. Thus
\[
 R=\bigoplus_{z\in Z}\OO z
\]
is a free $\OO$-module of finite rank $|Z|$. Since $\OO$ is
$J(\OO)$-adically complete and $Z$ is finite, coefficientwise
completeness gives an isomorphism
\[
 R\longrightarrow\varprojlim_{r\geq1}R/J(\OO)^rR.
\]
Hence $R$ is complete for the $J(\OO)R$-adic topology and,
equivalently, for the $J(R)$-adic topology.

Choose a set $T$ of representatives for the cosets of $Z$ in $N$.
Then $\OO N=\bigoplus_{t\in T}Rt$ is a free $R$-module of
finite rank $|N:Z|$. Multiplication by $b$ is an $R$-linear
idempotent endomorphism of $\OO N$, with image $\OO Nb$.
Consequently, $\OO Nb$ is a finitely generated projective
$R$-module, and hence is free because $R$ is local.
The block correspondence for a central
$p$-quotient gives
\[
 (\OO Nb)/J(R)(\OO Nb)
       \cong k(N/Z)q(b)\cong M_d(k).
\]
Thus $\OO Nb$ has $R$-rank $d^2$.

Lift the diagonal matrix units of $M_d(k)$ to orthogonal
idempotents $E_1,\ldots,E_d$ with sum $b$. Lifting the matrix
units in the corresponding corners gives
$x_i\in E_i(\OO Nb)E_1$ and $y_i\in E_1(\OO Nb)E_i$;
take $x_1=y_1=E_1$. The element $y_ix_i$ reduces to the identity
of its corner and is therefore invertible there. Replacing $y_i$
by $(y_ix_i)^{-1}y_i$, we may assume that $y_ix_i=E_1$.
The idempotent $E_i-x_iy_i$ lies in the Jacobson radical and
is consequently zero. Hence the elements $x_iy_j$ form a
system of matrix units.

Together with the central embedding of $R$, these matrix units
define an $R$-algebra homomorphism
$M_d(R)\longrightarrow\OO Nb$ whose reduction modulo $J(R)$
is an isomorphism. Nakayama's lemma makes this map surjective.
Since both modules are free of the same rank, the map is an
isomorphism. The same argument applies to $\OO Lb'$.
\end{proof}

\subsection{Central-character fibers}

For each $\lambda\in\Irr(Z)$, Lemma~\ref{dgn:lem:matrix-base}
shows that the fiber above $\lambda$ contains exactly one
irreducible character $\theta_\lambda$ in $b$ and exactly one
irreducible character $\phi_\lambda$ in $b'$. Their degrees are
$d$ and $e$, respectively. The relation with the defect-zero
quotient is given by
\begin{equation}\label{dgn:eq:gamma-fiber}
 \theta_\lambda(g)=
 \begin{cases}
 \lambda(g_p)\theta_0(gZ),&g_p\in Z,\\
 0,&g_p\notin Z.
 \end{cases}
\end{equation}
Here $g_p$ denotes the $p$-part of $g$. The analogous formula
holds for $\phi_\lambda$. This is the central linear-fiber
form of \cite[Lemma~5.6]{NS14}. On $p'$-elements, the formula
agrees with the inflation of $\theta_0$, and the degrees agree.

Now let $\lambda$ be the unique irreducible constituent of
$\theta_Z$. The generalized DGN correspondence respects the
central-character fibers \cite[proof of Theorem~5.7]{NS14}.
Consequently, $\theta=\theta_\lambda$ and
$\phi=\phi_\lambda$.

Formula~\eqref{dgn:eq:gamma-fiber} is compatible with group
conjugation and Galois automorphisms. Together with the
equivariant defect-zero DGN correspondence, it gives
\begin{equation}\label{dgn:eq:bottom-stabilizers}
 (H\times\HH)_\theta=(H\times\HH)_\phi.
\end{equation}
For the reverse implication, use the inverse defect-zero DGN
correspondence and the fact that both characters lie above the
same $\lambda$.

\subsection{The multiplicity algebra}

We now pass from the two block algebras to the algebra recording
the restriction multiplicity $n$. The Dade algebras used below
arise from the theory of endopermutation modules
\cite{Dad78a,Dad78b}; for general background on block algebras,
see \cite{Lin18}.

\begin{lemma}\label{dgn:lem:multiplicity-algebra}
Set $i=bb'$. The map $\OO Lb'\longrightarrow i\OO Ni$,
$v\longmapsto bv$, is an injective $R$-algebra homomorphism.
With respect to this embedding,
\[
 S=C_{i\OO Ni}(\OO Lb')\cong M_n(R).
\]
The reduction $\bar S=S/J(R)S$ is naturally the multiplicity
algebra of the defect-zero DGN quotient configuration. If
$P=D/Z$, then $\bar S$ is a Dade $P$-algebra with
$\bar S(P)=k$, and $p\nmid n$.
\end{lemma}

\begin{proof}
Write $\OO Nb=\End_R(V)$, where $V$ is free of rank $d$.
The matrix units of $\OO Lb'\cong M_e(R)$ give a decomposition
\[
 iV\cong R^e\otimes_R W_0,
\]
where $W_0$ is finitely generated projective over $R$, and hence
free since $R$ is local. After
specialization along the augmentation $R\to\OO$, the generic
fibers of $V$ and $iV$ afford the inflations of $\theta_0$ and
$n\phi_0$, respectively. Thus $W_0$ has rank $n$. The double
centralizer calculation gives
$S=\End_R(W_0)\cong M_n(R)$. Since $n>0$, the action of
$M_e(R)$ on $iV$ is faithful, proving that the local block
homomorphism is injective.

The matrix-unit decomposition also shows that formation of this
centralizer commutes with coefficient base change. Its reduction
is therefore the multiplicity algebra of the defect-zero quotient
configuration. By \cite[Lemma~11.9]{Lad11}, this is a Dade
$P$-algebra with $\bar S(P)=k$. Counting orbits on a
$P$-stable basis gives
\[
 n^2=\dim_k\bar S\equiv\dim_k\bar S(P)=1\pmod p,
\]
so $p\nmid n$.
\end{proof}

The restriction multiplicity is the same on every
central-character fiber. Indeed, substituting
\eqref{dgn:eq:gamma-fiber} and its local analogue into the
inner product cancels the common root-of-unity factors, and
the remaining sum descends to $L/Z$. Hence
\[
 [ (\theta_\lambda)_L,\phi_\lambda]
   =[(\theta_0)_{L/Z},\phi_0]=n.
\]
For the remainder of the paper, fix a system of matrix units for
$S\cong M_n(R)$. These matrix units need not be fixed by the
group or Galois actions.

\section{Semilinear realizations and integral lifting}\label{sec:lift}

In this section, we construct a semilinear realization of the mixed
stabilizer on the multiplicity algebra. We first construct a residual
realization and then lift it integrally using a determinant normalization.
Compatibility with blocks at all intermediate groups will be established
in the subsequent sections.

\subsection{The action of the mixed stabilizer}
\label{dgn:sec:central91-mixed}

We use the central-defect configuration of Section~\ref{sec:central}, with
$R=\OO Z$, $i=bb'$ and $S=C_{i\OO Ni}(\OO Lb')$, and the fixed matrix
form $S\cong M_n(R)$. Set
\[
 \mathcal G=\Gal(K/\Q_p),\qquad
 \Gamma=(H\times\mathcal G)_\theta,\qquad P=D/Z,
\]
where $K$ in $\Gal(K/\Q_p)$ is the coefficient field of the
$p$-modular system $(K,\OO,k)$.
For $a=(h,\gamma)\in\Gamma$, the natural right action on the group
algebra is given by
\begin{equation}\label{dgn:eq:c91-natural-action}
 \beta_{(h,\gamma)}\left(\sum c_g g\right)
          =\sum\gamma(c_g)h^{-1}gh.
\end{equation}
Let $\varepsilon_a$ denote the map on $M_n(R)$ that acts entrywise
through $\beta_a|_R$. Thus $\varepsilon_a$ fixes the chosen matrix units.

\subsection{The residual Dade algebra}

In the quotient configuration, bars denote the images induced by
coefficient reduction $\OO\to k$ and the quotient maps by $Z$.

\begin{lemma}\label{dgn:lem:c91-natural-residual}
Every $a\in\Gamma$ stabilizes $R$, $J(R)$ and $S$, and fixes $i$.
Consequently, $\beta$ restricts to an $R$-semilinear right action on
$S$ and induces a semilinear right action on $\bar S=S/J(R)S$.
There is a natural $\Gamma$-equivariant isomorphism
\begin{equation}\label{dgn:eq:c91-residual-centralizer}
 \bar S\cong
 C_{\bar i\,k(N/Z)\bar i}\bigl(k(L/Z)\bar b'\bigr).
\end{equation}
The algebra on the right is the multiplicity algebra of the defect-zero
quotient configuration. It is a Dade $P$-algebra, and the scalar map
identifies its Brauer quotient with $k$:
\begin{equation}\label{dgn:eq:c91-residual-brauer}
 \bar S(P)\cong k,\qquad
 \operatorname{Br}_P(z1)=z\quad(z\in k).
\end{equation}
\end{lemma}

\begin{proof}
Let $a=(h,\gamma)\in\Gamma$. By
\eqref{dgn:eq:c91-natural-action} and the formula for the central
idempotent associated with an ordinary irreducible character,
$\beta_a(e_\theta)=e_{\theta^a}=e_\theta$.
Since group and Galois actions transport blocks, it follows that
$\beta_a(b)=b$. Equivariance of the generalized DGN correspondence
gives $\phi^a=\phi$ and hence $\beta_a(b')=b'$.
More explicitly, $\theta$ is determined by its central character
$\lambda\in\Irr(Z)$ and the corresponding defect-zero character
$\theta_0$ of the quotient. A mixed element fixing $\theta$ fixes
both $\lambda$ and $\theta_0$. Equivariance of the defect-zero DGN
correspondence then fixes the correspondent of $\theta_0$, and the
relative defect-zero fiber formula shows that $\phi$ is fixed; see
\cite[Lemma~5.6 and the proof of Theorem~5.7]{NS14}.

As $h$ normalizes $D$ and $N\nrm A$, it stabilizes
$Z=N\cap D$ and $L=N_N(D)$. The Galois action preserves $J(\OO)$,
while conjugation by $h$ sends $z-1$ to $h^{-1}zh-1$.
Thus $R$ and $J(R)$ are stable, and $i=bb'$ is fixed.
The algebra $\OO Lb'$ is also stable. Its centralizer in $i\OO Ni$,
namely $S$, is therefore stable.

To prove \eqref{dgn:eq:c91-residual-centralizer}, use the matrix-unit
decomposition underlying the integral matrix structure from Section~\ref{sec:central}:
\[
 \OO Nb=\End_R(V),\qquad iV\cong R^e\otimes_R W,
 \qquad \OO Lb'\cong M_e(R),\qquad S\cong\End_R(W).
\]
After base change along $R\to k$, the local matrix algebra acts on
the first tensor factor, and its centralizer is
$\End_k(W/J(R)W)$. The reduction map from $S$ therefore induces
the asserted isomorphism. This map is $\Gamma$-equivariant because
it is induced by the group-algebra quotient map.

Finally, $L/Z=C_{N/Z}(D/Z)$. Indeed, an element $x\in N$ normalizes
$D$ if and only if its commutators with all $d\in D$ lie in
$N\cap D=Z$, or equivalently, if $xZ$ centralizes $D/Z$.
The quotient blocks $\bar b$ and $\bar b'$ are corresponding
defect-zero blocks. The Dade-algebra argument in
\cite[Lemma~11.9]{Lad11} now gives
\eqref{dgn:eq:c91-residual-brauer}. If this result is first applied
over a smaller finite field, extending scalars preserves the permutation
basis and commutes with formation of the Brauer quotient.
\end{proof}

\subsection{Residual crossed representations}

We use right actions throughout. Products of operators are composed
as $TU=T\circ U$. In particular, the entrywise field operators satisfy
$F_\tau F_\upsilon=F_{\tau\circ\upsilon}$.

\begin{lemma}\label{dgn:lem:c91-full-residual}
Let $\bar S=M_n(k)$ be a Dade $P$-algebra over a finite field $k$,
with $\bar S(P)=k$ identified by the scalar map in
\eqref{dgn:eq:c91-residual-brauer}. Suppose that a finite group
$\Gamma$ acts semilinearly on the right on $\bar S$ and normalizes
the image of $P$ in $\Aut_k(\bar S)$. Write $\bar\beta_a$ for this
action and $\tau_a$ for its restriction to $k$. Relative to the fixed
matrix form, there is a map $\bar\sigma:\Gamma\to M_n(k)^\times$
satisfying
\begin{align}
 \bar\beta_a(X)
   &=\bar\sigma(a)^{-1}\tau_a(X)\bar\sigma(a),
       \label{dgn:eq:c91-residual-inner}\\
 \tau_{b_0}\bigl(\bar\sigma(a)\bigr)\bar\sigma(b_0)
   &=\bar\sigma(ab_0).
       \label{dgn:eq:c91-residual-crossed}
\end{align}
Here $\tau_a$ acts entrywise. Moreover, $\bar\sigma(a)=1$ whenever
$a$ acts trivially on $\bar S$.
\end{lemma}

\begin{proof}
Let $V=k^n$, and regard $\bar S$ as a subalgebra of
$\End_{\mathbb F_p}(V)$. There is a unique homomorphism
$\rho:P\to\bar S^\times$ realizing the right $P$-action:
\[
 \bar\beta_u(X)=\rho(u)^{-1}X\rho(u)\qquad(u\in P).
\]
Indeed, existence follows from $H^2(P,k^\times)=1$, and uniqueness
from $\operatorname{Hom}(P,k^\times)=1$, since $|P|$ and
$|k^\times|$ are coprime. Equivalently, one can apply the
Schur--Zassenhaus theorem to the corresponding finite central extension.

Apply \cite[Theorem~4.2]{Fu26} to $\bar S$ and its Brauer quotient $k$,
using the action of $k$ on itself by multiplication. We obtain
a homomorphism
\begin{equation}\label{dgn:eq:c91-fu-fusion}
 \Phi:N_{\operatorname{GL}_{\mathbb F_p}(V)}(\rho(P))
       \longrightarrow\operatorname{GL}_{\mathbb F_p}(k),
 \qquad
 \Phi(c)=\operatorname{Br}_P(c)\quad
       \bigl(c\in(\bar S^P)^\times\bigr).
\end{equation}
Here the natural algebra embeddings are suppressed. In particular,
if $m_z$ denotes multiplication by $z$, then
\begin{equation}\label{dgn:eq:c91-fusion-scalars}
 \Phi(m_z)=m_z\qquad(z\in k^\times).
\end{equation}
Fix such a homomorphism $\Phi$.

For $a\in\Gamma$, the semilinear form of the Skolem--Noether theorem
provides an invertible $\mathbb F_p$-linear operator $T_a$ such that
\begin{equation}\label{dgn:eq:c91-right-implementer}
 \bar\beta_a(X)=T_a^{-1}XT_a.
\end{equation}
Under this convention, $T_a$ is $\tau_a^{-1}$-semilinear: taking
$X=m_z$ gives
\[
 T_a^{-1}m_zT_a=m_{\tau_a(z)}.
\]
To see that $T_a$ belongs to the domain of $\Phi$, consider the
quotient map $\bar S^\times\to\bar S^\times/k^\times$.
The subgroup $\rho(P)$ is the unique Sylow $p$-subgroup of the
inverse image of its image under this map: that inverse image is
$k^\times\rho(P)$, with $k^\times\cap\rho(P)=1$.
Since $\bar\beta_a$ normalizes the image of the $P$-action, it
stabilizes this inverse image and hence $\rho(P)$. Thus
$T_a\in N_{\operatorname{GL}_{\mathbb F_p}(V)}(\rho(P))$.

Applying $\Phi$ to the scalar identity above and using
\eqref{dgn:eq:c91-fusion-scalars}, we obtain
\[
 \Phi(T_a)^{-1}m_z\Phi(T_a)=m_{\tau_a(z)}.
\]
It follows that there is a unique $c_a\in k^\times$ such that
\[
 \Phi(T_a)=m_{c_a}F_{\tau_a^{-1}}.
\]
In fact, $c_a=\Phi(T_a)(1)$, and the semilinear relation gives
$\Phi(T_a)(z)=\tau_a^{-1}(z)c_a$.
Define $\widehat T_a=m_{c_a}^{-1}T_a$. Then
\begin{equation}\label{dgn:eq:c91-normalized-implementer}
 \Phi(\widehat T_a)=F_{\tau_a^{-1}}.
\end{equation}
Any two operators realizing $\bar\beta_a$ differ by a scalar in
$k^\times$. By \eqref{dgn:eq:c91-fusion-scalars}, there is therefore
a unique implementing operator satisfying
\eqref{dgn:eq:c91-normalized-implementer}.

Since $\bar\beta_{ab_0}=\bar\beta_{b_0}\circ\bar\beta_a$, the product
$\widehat T_a\widehat T_{b_0}$ realizes $\bar\beta_{ab_0}$.
Furthermore, $\tau_{ab_0}=\tau_{b_0}\circ\tau_a$, so
\[
 \Phi(\widehat T_a\widehat T_{b_0})
  =F_{\tau_a^{-1}}F_{\tau_{b_0}^{-1}}
  =F_{\tau_{ab_0}^{-1}}.
\]
The uniqueness of the normalized implementer yields
\begin{equation}\label{dgn:eq:c91-right-homomorphism}
 \widehat T_a\widehat T_{b_0}=\widehat T_{ab_0}.
\end{equation}
Write
\[
 \widehat T_a=F_{\tau_a^{-1}}\bar\sigma(a),\qquad
 \bar\sigma(a)=F_{\tau_a}\widehat T_a\in\operatorname{GL}_n(k).
\]
Then \eqref{dgn:eq:c91-right-implementer} gives
\eqref{dgn:eq:c91-residual-inner}. Substituting this expression into
\eqref{dgn:eq:c91-right-homomorphism} and using
$X F_{\tau^{-1}}=F_{\tau^{-1}}\tau(X)$ gives
\eqref{dgn:eq:c91-residual-crossed}.
If $a$ acts trivially on $\bar S$, then $\tau_a=1$ and $T_a$ is
scalar. Hence $\widehat T_a=1$ and $\bar\sigma(a)=1$.
\end{proof}

\begin{remark}
The homomorphism $\Phi$ is applied to the implementing operators
$T_a$. The pure field operator $F_\tau$ need not normalize $\rho(P)$,
and the argument does not require $\Phi(F_\tau)$ to be defined.
The resulting realization depends on the fixed choice of $\Phi$;
different fusion homomorphisms need not yield the same realization.
\end{remark}

\subsection{A normalized integral lifting theorem}

We now lift the residual realization to $R$. The proof uses innerness
of matrix-algebra automorphisms over a local ring, followed by an
explicit determinant normalization. The innerness of $R$-linear
automorphisms is a special case of the results of Rosenberg and
Zelinsky on automorphisms of separable algebras \cite{RZ61}.
For the matrix-algebra and Morita-theoretic background, see
\cite{Lin18}. Related lifting methods for blocks and source
algebras were developed by K\"ulshammer, Okuyama and Watanabe
\cite{KOW00}. Here we give the lifting argument directly for $S$
and its prescribed residual realization.

\begin{theorem}\label{dgn:thm:c91-integral-lift}
Let $R$ be a commutative local ring, complete in the $J(R)$-adic
topology, with finite residue field $k$ of characteristic $p$.
Assume that $p\nmid n$. Suppose that a finite group $\Gamma$ acts
semilinearly on the right on $S=M_n(R)$, with action $\beta_a$.
Let $\varepsilon_a$ act entrywise through $\beta_a|_R$, relative to
this fixed matrix form. Fix a residual realization $\bar\sigma$
satisfying
\eqref{dgn:eq:c91-residual-inner}--\eqref{dgn:eq:c91-residual-crossed}.
Then there is a unique map $\sigma:\Gamma\to S^\times$ satisfying
\begin{align}
 \beta_a(X)&=\sigma(a)^{-1}\varepsilon_a(X)\sigma(a),
        \label{dgn:eq:c91-integral-inner}\\
 \varepsilon_{b_0}(\sigma(a))\sigma(b_0)&=\sigma(ab_0),
        \label{dgn:eq:c91-integral-crossed}\\
 \sigma(a)\bmod J(R)&=\bar\sigma(a),
        \label{dgn:eq:c91-integral-reduction}\\
 \det\sigma(a)&=t(\det\bar\sigma(a)),
        \label{dgn:eq:c91-integral-determinant}
\end{align}
where $t:k^\times\to R^\times$ is the Teichm\"uller section.
\end{theorem}

\begin{proof}
Let $q=|k|$. For each $z\in k^\times$, Hensel's lemma gives a unique
lift $t(z)\in R^\times$ satisfying $t(z)^{q-1}=1$.
By uniqueness, these lifts define a multiplicative,
$\Gamma$-equivariant section $t:k^\times\to R^\times$.
Since $p\nmid n$, the integer $n$ is a unit in $R$. A further
application of Hensel's lemma shows that
\[
 1+J(R)\longrightarrow1+J(R),\qquad v\longmapsto v^n
\]
is an automorphism. Its inverse, the unique $n$th-root map on
$1+J(R)$, is also $\Gamma$-equivariant.

For each $a\in\Gamma$, the automorphism
$\beta_a\circ\varepsilon_a^{-1}$ is $R$-linear and hence inner,
since $R$ is a commutative local ring. We may therefore choose
$t_a\in S^\times$ such that
\[
 \beta_a(X)=t_a^{-1}\varepsilon_a(X)t_a.
\]
The reductions $\bar t_a$ and $\bar\sigma(a)$ realize the same
semilinear automorphism and thus differ by a scalar in $k^\times$.
Multiplying $t_a$ by a scalar unit in $R$, we may assume that
$\bar t_a=\bar\sigma(a)$.
The right-action identities for $\beta$ and $\varepsilon$ imply
that there are scalars $c(a,b_0)\in R^\times$ with
\[
 \varepsilon_{b_0}(t_a)t_{b_0}=c(a,b_0)t_{ab_0}.
\]
By the residual crossed identity, $c(a,b_0)\in1+J(R)$.

Set $d(a)=\det t_a$ and define
\[
 d_1(a)=d(a)\,t(\det\bar\sigma(a))^{-1}\in1+J(R).
\]
Taking determinants in the preceding identity, and using the residual
crossed identity together with the multiplicativity and equivariance
of $t$, gives
\[
 \varepsilon_{b_0}(d_1(a))d_1(b_0)
       =c(a,b_0)^n d_1(ab_0).
\]
Let $u(a)$ be the unique $n$th root of $d_1(a)$ in $1+J(R)$.
Taking the unique $n$th roots within this group yields
\[
 \varepsilon_{b_0}(u(a))u(b_0)=c(a,b_0)u(ab_0).
\]
It follows that $\sigma(a)=u(a)^{-1}t_a$ satisfies the crossed
identity \eqref{dgn:eq:c91-integral-crossed}.
Since $u(a)$ is scalar and congruent to $1$ modulo $J(R)$, this
adjustment preserves the inner-action and reduction identities.
Finally, $u(a)^n=d_1(a)$ gives
$\det\sigma(a)=t(\det\bar\sigma(a))$.

For uniqueness, suppose that $\sigma'$ satisfies the same four
identities. Since $\sigma'(a)$ and $\sigma(a)$ realize the same
automorphism, $\sigma'(a)=v(a)\sigma(a)$ for some $v(a)\in R^\times$.
Equality of reductions gives $v(a)\in1+J(R)$, while equality of
determinants gives $v(a)^n=1$. The uniqueness of $n$th roots in
$1+J(R)$ implies $v(a)=1$.
\end{proof}

\begin{remark}
The uniqueness in Theorem~\ref{dgn:thm:c91-integral-lift} is relative
to the fixed matrix form, the action $\beta$, the residual realization
$\bar\sigma$ and the prescribed determinant normalization.
It does not assert uniqueness of the fusion homomorphism or identify
lifts of different residual realizations. Nor does it establish an
identification with a separately constructed character correspondence.
\end{remark}

\subsection{Normalization on centralizers}

We apply the two preceding constructions to the mixed stabilizer and
record the normalization on $L$ and $C_A(M)$.

\begin{proposition}\label{dgn:prop:c91-full-integral-mixed}
In the central-defect configuration, the fixed matrix form
$S\cong M_n(R)$ admits a realization $\sigma$ of the full group
$\Gamma$ satisfying
\eqref{dgn:eq:c91-integral-inner}--\eqref{dgn:eq:c91-integral-determinant}.
For pure group elements, write $\sigma(h)=\sigma((h,1))$ whenever
$(h,1)\in\Gamma$. Then
\[
 \sigma(l)=1\quad(l\in L),\qquad
 \sigma(c)=1\quad(c\in C_A(M)).
\]
\end{proposition}

\begin{proof}
Lemma~\ref{dgn:lem:c91-natural-residual} provides the semilinear
$\Gamma$-action and the residual Dade algebra. The group part
normalizes $D/Z$, and field automorphisms commute with group
conjugation. Hence $\Gamma$ normalizes the image of the $P$-action,
and Lemma~\ref{dgn:lem:c91-full-residual} gives a residual realization
of the full group $\Gamma$ in the fixed matrix form.
Since $R=\OO Z$ is complete, commutative and local, with finite
residue field $k$, and $p\nmid n$, the integral realization follows
from Theorem~\ref{dgn:thm:c91-integral-lift}.

Let $l\in L$. Then $(l,1)\in\Gamma$, since inner automorphisms
from $N$ fix $\theta$. By the definition of $S$, conjugation by $l$
acts trivially on $S$. It also acts trivially on $R$, since
$Z\leq Z(N)$. Lemma~\ref{dgn:lem:c91-full-residual} therefore gives
$\bar\sigma(l)=1$. The inner-action identity implies that
$\sigma(l)=v_l1_S$ for some $v_l\in R^\times$. The reduction and
determinant identities then give
\[
 v_l\in1+J(R),\qquad v_l^n=1.
\]
The uniqueness of $n$th roots in $1+J(R)$ yields $v_l=1$, and hence
$\sigma(l)=1$.

If $c\in C_A(M)$, then $c$ centralizes both $N$ and $D$. Thus
$c\in H$ and $(c,1)\in\Gamma$, and $c$ acts trivially on both
$S\subseteq\OO N$ and $R$. The same argument gives $\sigma(c)=1$.
\end{proof}

\begin{remark}
The residual normalization above is obtained on
$S/J(R)S\cong M_n(k)$. By contrast, reduction modulo $J(\OO)$ gives
$S/J(\OO)S\cong M_n(kZ)$, which still contains the nilpotent
directions of $kZ$. Brauer normalization on this latter algebra and
block compatibility for all intermediate groups require further
arguments; they do not follow from the proposition alone.
\end{remark}

\section{Graded corner correspondences}\label{sec:corner}

We use the integral realization constructed in Section~\ref{sec:lift}
to obtain a graded corner isomorphism. We then specialize this
isomorphism to a central-character fiber and construct compatible
character correspondences for intermediate groups. The resulting
correspondences preserve relative degrees and, on the groups occurring
in Theorem~A, character heights.

\subsection{The corner algebra and the crossed action}
\label{dgn:subsec:c91-corner}

In the setting of Theorem~A, we take $G=A_\theta$ and
$H_0=H_\theta$. Lemma~\ref{dgn:lem:central91-frattini} gives
$G=NH_0$. For the algebraic construction in Subsections~\ref{dgn:subsec:c91-corner}--\ref{dgn:subsec:c91-specialization},
the letter $A$ denotes the base corner algebra. We use the convention
$u^h=h^{-1}uh$ and assume that
\[
 \begin{gathered}
 G=NH_0,\qquad N\trianglelefteq G,\qquad L=N\cap H_0,\\
 Z\leq L\cap Z(N),\qquad Z\trianglelefteq G,
 \qquad R=\OO Z.
 \end{gathered}
\]
Here $\OO$ is the complete discrete valuation ring of our splitting
$p$-modular system, and $Z$ is a finite abelian $p$-group. Let
$b\in Z(\OO N)$ be a $G$-invariant block idempotent and let
$b'\in Z(\OO L)$ be an $H_0$-invariant block idempotent. Set
\[
 B=\OO G b,\qquad B_0=\OO N b,\qquad i=bb',
 \qquad A=iB_0i,\qquad T=\OO Lb'.
\]
Since $b$ is central in $\OO N$, the element $i$ is an idempotent.
It is fixed by $H_0$, although it need not be fixed by $G$.
The images of $R$ in $B$ and $B_0$ are $Rb$, whereas its images
in $A$ and $T$ are $Ri$ and $Rb'$, respectively. We suppress these
identity idempotents when no ambiguity arises.

Assume, as holds in the central-defect configuration, that
$B_0\cong M_d(R)$, $T\cong M_e(R)$ and $iB_0i\ne0$.
The local block algebra acts on the corner through
\begin{equation}\label{dgn:eq:c91-base-embedding}
 \iota:T\longrightarrow A,\qquad t\longmapsto bt.
\end{equation}
Its centralizer is the multiplicity algebra
\[
 S=C_A(\iota(T))\cong M_n(R),\qquad n>0.
\]
All these matrix algebra isomorphisms preserve the embedded copies of
$R$. Fix $R$-matrix units $E_{rs}$ in $S$ and set
\[
 S_0=\sum_{r,s}\OO E_{rs}\cong M_n(\OO),\qquad j=E_{11}.
\]
For $h\in H_0$, let $\varepsilon_h$ act on matrix entries through
conjugation by $h$ on $R$, fixing the chosen matrix units. Suppose
that elements $\sigma_h=\sigma(h)\in S^\times$ satisfy
\begin{align}
 h^{-1}sh&=\sigma_h^{-1}\varepsilon_h(s)\sigma_h,
                  \label{dgn:eq:c91-action}\\
 \varepsilon_k(\sigma_h)\sigma_k&=\sigma_{hk},
                  \label{dgn:eq:c91-crossed}\\
 \sigma_l&=i\qquad(l\in L),
                  \label{dgn:eq:c91-trivial-L}
\end{align}
for all $h,k\in H_0$ and $s\in S$. In particular,
\eqref{dgn:eq:c91-crossed} with $h=k=1$ gives $\sigma_1=i$.
In the configuration of Theorem~A, these elements are supplied by
Section~\ref{sec:lift}.

\begin{lemma}\label{dgn:lem:c91-base-centralizer}
The homomorphism \eqref{dgn:eq:c91-base-embedding} is injective, and
\begin{equation}\label{dgn:eq:c91-base-centralizer}
 C_A(S_0)=C_A(S)=\iota(T).
\end{equation}
\end{lemma}

\begin{proof}
Identify $B_0$ with $\End_R(V)$, where $V$ is free of rank $d$.
The module $iV$ is a nonzero direct summand of $V$. Since $R$ is
local, $iV$ is free, and $A=\End_R(iV)$.
Choose matrix units $F_{uv}$ for $T\cong M_e(R)$, and use the same
notation for their images under $\iota$. Their diagonal sum acts as
the identity on $iV$. With $W=F_{11}(iV)$, the matrix-unit identities
give
\[
 iV\cong R^e\otimes_R W,\qquad
 \iota(T)=\End_R(R^e)\otimes1_W.
\]
The module $W$ is nonzero and finitely generated projective over $R$;
it is therefore free, since $R$ is local. Hence $T$ acts faithfully
on $iV$, proving injectivity. The same decomposition yields
\[
 S=1\otimes\End_R(W),\qquad
 C_A(S)=\End_R(R^e)\otimes1_W=\iota(T).
\]
Finally, $S=RS_0$ and $R$ is central in $A$, so an element of $A$
centralizes $S_0$ if and only if it centralizes $S$.
\end{proof}

\subsection{The graded corner isomorphism}

A full idempotent yields a Morita equivalence through the standard
corner construction; see \cite{Lin18}. We give an explicit graded
algebra isomorphism, following the magic-representation approach of
\cite{Lad11}.

\begin{proposition}\label{dgn:prop:c91-graded-corner}
Under the preceding assumptions, define
\begin{equation}\label{dgn:eq:c91-kappa}
 \kappa(h)=hi\,\sigma_h^{-1}\in(iBi)^\times
                       \qquad(h\in H_0),
\end{equation}
where $hi$ is viewed as a unit of the corner algebra with identity
$i$. Then $\kappa$ is a group homomorphism and induces
$\OO$-algebra isomorphisms
\begin{equation}\label{dgn:eq:c91-corner-isomorphism}
 \OO H_0b'
   \xrightarrow{\ \kappa\ }C_{iBi}(S_0)
   \xrightarrow{\ c\mapsto jcj\ }jBj.
\end{equation}
These isomorphisms preserve the embedded copies of $R$ and the
grading by $G/N\cong H_0/L$. Moreover, $i$ and $j$ are full
idempotents in $B$:
\begin{equation}\label{dgn:eq:c91-fullness}
 BiB=B,\qquad BjB=B.
\end{equation}
Consequently, \eqref{dgn:eq:c91-corner-isomorphism} yields an integral
Morita equivalence between $B$ and $\OO H_0b'$. This construction
does not require $i$ or $j$ to be $G$-invariant.
\end{proposition}

\begin{proof}
All inverses below are taken in the corner algebra with identity $i$.
Since $i$ is $H_0$-invariant, equations
\eqref{dgn:eq:c91-action} and \eqref{dgn:eq:c91-crossed} give
\begin{align*}
 \kappa(h)\kappa(k)
 &=hi\,\sigma_h^{-1}ki\,\sigma_k^{-1}\\
 &=hki\,(\sigma_h^{-1})^k\sigma_k^{-1}\\
 &=hki\,\sigma_k^{-1}\varepsilon_k(\sigma_h^{-1})\\
 &=hki\,(\varepsilon_k(\sigma_h)\sigma_k)^{-1}
   =\kappa(hk).
\end{align*}
Thus $\kappa$ is a group homomorphism. For $s\in S$, we also have
\begin{equation}\label{dgn:eq:c91-kappa-centralizes}
 \kappa(h)^{-1}s\kappa(h)
   =\sigma_hh^{-1}sh\sigma_h^{-1}=\varepsilon_h(s).
\end{equation}
Since $\varepsilon_h$ fixes $S_0$ pointwise, $\kappa(h)$ centralizes
$S_0$. Equation~\eqref{dgn:eq:c91-trivial-L} further gives
$\kappa(l)=li$ for $l\in L$.

The universal property of the group algebra extends $\kappa$ uniquely
to a unital $\OO$-algebra homomorphism
$\OO H_0\longrightarrow C_{iBi}(S_0)$. Its restriction to $\OO L$
is $u\mapsto bub'$. In particular, it sends $b'$ to $i$ and
$1-b'$ to zero, and hence factors through the first map in
\eqref{dgn:eq:c91-corner-isomorphism}. On $T$, the induced map is
precisely \eqref{dgn:eq:c91-base-embedding}.

To prove that this map is an isomorphism, we compare homogeneous
components. Choose a transversal $\mathcal T$ for $H_0/L$.
Since $G=NH_0$, it is also a transversal for $G/N$, and
\begin{equation}\label{dgn:eq:c91-corner-grading}
 B=\bigoplus_{h\in\mathcal T}B_0h,\qquad
 iBi=\bigoplus_{h\in\mathcal T}iB_0hi
     =\bigoplus_{h\in\mathcal T}Ah.
\end{equation}
The last equality follows from the $H_0$-invariance of $i$.
As $h$ normalizes $A$ and $\sigma_h\in A^\times$, we have
\[
 Ah=Ahi\,\sigma_h^{-1}=A\kappa(h).
\]
The algebra $S_0$ lies in the identity component. Therefore an
element centralizes $S_0$ exactly when each of its homogeneous
components does. By \eqref{dgn:eq:c91-kappa-centralizes} and
Lemma~\ref{dgn:lem:c91-base-centralizer},
\begin{equation}\label{dgn:eq:c91-centralizer-grading}
 C_{iBi}(S_0)
       =\bigoplus_{h\in\mathcal T}\iota(T)\kappa(h).
\end{equation}
On the other hand, $\OO H_0b'=\bigoplus_{h\in\mathcal T}Th$.
On each pair of corresponding components, the induced map is
$th\mapsto\iota(t)\kappa(h)$. This is a bijection because $\iota$
is injective and $\kappa(h)$ is invertible. Thus the first map in
\eqref{dgn:eq:c91-corner-isomorphism} is a graded algebra isomorphism.

Write $C=C_{iBi}(S_0)$. Since every element of $C$ commutes with
$j\in S_0$, compression defines a unital algebra homomorphism
\[
 f:C\longrightarrow jBj,\qquad c\longmapsto jcj.
\]
Its inverse is
\begin{equation}\label{dgn:eq:c91-corner-inverse}
 g:jBj\longrightarrow C,\qquad
 t\longmapsto\sum_{r=1}^nE_{r1}tE_{1r}.
\end{equation}
Indeed, for all $u,v$,
$E_{uv}g(t)=E_{u1}tE_{1v}=g(t)E_{uv}$, so $g(t)\in C$.
The matrix-unit identities give
$g(t)g(t')=g(tt')$, $f(g(t))=t$ and $g(f(c))=c$.
All the matrix units lie in the identity component, so both maps
preserve the grading. They also preserve the embedded copies of $R$:
the first isomorphism sends $rb'$ to $ri$, and compression sends
$ri$ to $rj$.

It remains to prove fullness. In $B_0=\End_R(V)$, both summands of
$V=iV\oplus(1-i)V$ are free, and $iV\ne0$. A basis adapted to this
decomposition identifies $i$ with a diagonal idempotent having at
least one diagonal entry equal to $1$. Matrix units then give
$b\in B_0iB_0$, and hence $BiB=B$. Moreover,
\[
 i=\sum_{r=1}^nE_{r1}jE_{1r}\in BjB.
\]
Together with $b\in B_0iB_0$, this implies $b\in BjB$ and thus
$BjB=B$. The asserted Morita equivalence follows from the fullness
of $j$ and the isomorphism $\OO H_0b'\cong jBj$.
\end{proof}

\begin{remark}\label{dgn:rem:c91-noncentral-R}
The ring $R$ is central in $B_0$, $A$ and $S$, but need not be
central in $B$, since $H_0$ may act nontrivially on $Z$.
Accordingly, \eqref{dgn:eq:c91-corner-isomorphism} consists of
$\OO$-algebra isomorphisms preserving the specified copies of $R$;
the full algebras need not be central $R$-algebras.
\end{remark}

\subsection{Specialization to central-character fibers}\label{dgn:subsec:c91-specialization}

\begin{corollary}\label{dgn:cor:c91-specialization}
Suppose that $\lambda\in\Irr(Z)$ is $G$-invariant and has values
in $\OO$. Since $Z$ is abelian, $\lambda$ is a linear character.
We also denote by $\lambda$ its $\OO$-linear extension, which is
an $\OO$-algebra homomorphism
\[
 \lambda:R=\OO Z\longrightarrow\OO,\qquad
 \sum_{z\in Z}a_z z\longmapsto\sum_{z\in Z}a_z\lambda(z).
\]
Set $I_\lambda=\ker\lambda$ and $B_\lambda=B/I_\lambda B$.
Then $I_\lambda B=BI_\lambda$ is a two-sided ideal, and
\eqref{dgn:eq:c91-corner-isomorphism} induces an isomorphism
\begin{equation}\label{dgn:eq:c91-specialization}
 (\OO H_0b')/I_\lambda(\OO H_0b')
       \xrightarrow{\sim}jB_\lambda j.
\end{equation}
This isomorphism preserves the grading, and the image of $j$ is
full in $B_\lambda$. For every $N\leq U\leq G$, the construction
restricts to an isomorphism
\[
 \OO(U\cap H_0)b'\cong j\OO U b j,
\]
and this restriction is compatible with specialization.
\end{corollary}

\begin{proof}
Since $\lambda$ is $G$-invariant, the ideal $I_\lambda$ is
$G$-stable. The relation $gr=r^{g^{-1}}g$, for $g\in G$ and
$r\in R$, gives $BI_\lambda=I_\lambda B$. Thus $I_\lambda B$
is a two-sided ideal. The same argument applies to the source algebra.

The composite isomorphism of
Proposition~\ref{dgn:prop:c91-graded-corner} preserves the specified
copies of $R$, so it maps the source ideal onto $I_\lambda jBj$.
Since $j$ commutes with $R$,
\[
 I_\lambda jBj=j(I_\lambda B)j,\qquad
 (jBj)/j(I_\lambda B)j\cong jB_\lambda j.
\]
This proves \eqref{dgn:eq:c91-specialization}. The fullness of the
image of $j$ follows by passing to the quotient in $BjB=B$.
The grading also descends, since $I_\lambda$ lies in the identity
component. This argument does not require centralizers to commute
with the possibly nonflat specialization $R\to\OO$.

If $N\leq U\leq G$, then $U=N(U\cap H_0)$. Taking the direct
sums over the components indexed by $U/N$ in
\eqref{dgn:eq:c91-corner-grading} and
\eqref{dgn:eq:c91-centralizer-grading} gives the restriction
isomorphism. It uses the same matrix units $E_{rs}$, idempotent $j$
and elements $\sigma_h$, and therefore commutes with specialization.
\end{proof}

For the application to Theorem~A, take $G=A_\theta$ and
$H_0=H_\theta$. The identity $\theta_Z=\theta(1)\lambda$ shows
that $\lambda$ is $G$-invariant, so the corollary applies.

\subsection{Correspondences for intermediate groups}

We now return to the group notation of Theorem~A, with $G=A_\theta$,
$H_0=H_\theta$ and the fixed central character $\lambda$.
Extend scalars to the coefficient field $K$ of the splitting
$p$-modular system $(K,\OO,k)$. For each intermediate group
$N\leq U\leq G$, set $U'=U\cap H=U\cap H_0$.
The corner isomorphism on the $\lambda$-fiber induces a bijection
\[
 \Delta_U:\Irr(U\mid\theta)\longrightarrow\Irr(U'\mid\phi).
\]
To describe it explicitly, let $\pi:K[U]\to\End_K(V)$ afford
$\chi\in\Irr(U\mid\theta)$, and set
$V_i=\pi(i)V$ and $W=\pi(j)V_i$. Since $\kappa(U')$ centralizes
$S_0$, it commutes with $j$ and preserves $W$. The representation
\[
 U'\longrightarrow\GL_K(W),\qquad
 x\longmapsto\pi(\kappa(x))|_W
\]
affords $\Delta_U(\chi)$.

On this fiber, the base corner algebra is
$M_n(K)\otimes_K M_e(K)$. Consequently,
\begin{equation}\label{dgn:eq:degree-ratio}
 \frac{\chi(1)}{\Delta_U(\chi)(1)}
       =\frac de=\frac{\theta(1)}{\phi(1)}.
\end{equation}
Indeed, if $\chi_N=m\theta$, then
$\dim_K V=md$, $\dim_K V_i=mne$ and $\dim_K W=me$.
Thus $\Delta_U$ preserves relative degrees.

\subsection{Relative degrees and height-zero characters}

Write $\Delta_\theta=\Delta_M$.

\begin{lemma}\label{dgn:lem:central91-height-fibers}
The map $\Delta_\theta$ restricts to a bijection
\[
 \Delta_\theta:\Irr_0(M\mid\theta)
       \longrightarrow\Irr_0(M'\mid\phi).
\]
Every $\xi\in\Irr_0(M\mid\theta)$ is an extension of $\theta$,
and every $\eta\in\Irr_0(M'\mid\phi)$ is an extension of $\phi$.
\end{lemma}

\begin{proof}
The central-defect degree formulas give
\[
 \theta(1)_p=|N:Z|_p,\qquad \phi(1)_p=|L:Z|_p.
\]
Since $M=ND$, $M'=LD$ and $N\cap D=L\cap D=Z$, we also have
\[
 |M:D|_p=|N:Z|_p,\qquad |M':D|_p=|L:Z|_p.
\]
Let $B$ be the block of $M$ from Theorem~A, and let $B_0$ be the
unique block of $M'$ covering $b'$. In the generalized DGN
configuration, $B_0$ is the Brauer correspondent of $B$, and both
blocks have defect group $D$. By uniqueness of the covering blocks, each
$\xi\in\Irr(M\mid\theta)$ lies in $B$, and its image under
$\Delta_\theta$ lies in $B_0$. Hence equation~\eqref{dgn:eq:degree-ratio} gives
\[
 \htc(\xi)
 =v_p\!\left(\frac{\xi(1)}{\theta(1)}\right)
 =v_p\!\left(\frac{\Delta_\theta(\xi)(1)}{\phi(1)}\right)
 =\htc(\Delta_\theta(\xi)).
\]
Thus $\Delta_\theta$ preserves heights and restricts to the stated
bijection on height-zero characters.

The extension assertion is the fact used at the beginning of the
proof of \cite[Theorem~5.13]{NS14}. For completeness, it follows
from Clifford theory as follows. Since $\theta$ is $M$-invariant,
we have $\xi_N=e\theta$ for some positive integer $e$. This
multiplicity is the degree of an irreducible projective
representation of the $p$-group $M/N$, and is therefore a power of
$p$. If $\xi$ has height zero, then
$v_p(e)=v_p(\xi(1)/\theta(1))=0$, so $e=1$. Thus $\xi$ extends
$\theta$. The same argument, applied to $M'/L$, shows that every
$\eta\in\Irr_0(M'\mid\phi)$ extends $\phi$.

In particular, restriction to $N$ gives, for each such $\xi$,
\[
 A_\xi\leq A_\theta,\qquad
 (H\times\HH)_\xi\leq(H\times\HH)_\theta.
\]
Consequently, the construction on the fixed $\theta$-fiber applies
to every mixed stabilizer element needed for the final relation.
\end{proof}

\section{Mixed comparison functions}\label{sec:mixed}

We prove that the corner correspondence is equivariant under the mixed
stabilizer and construct associated projective representations with
matching comparison functions. We then show that the tensor
correspondence defined by this pair agrees with the corner correspondence
for every intermediate group under consideration.

\subsection{Mixed covariance on a fixed fiber}
\label{dgn:sec:synchronization}

We use the maps $\sigma$ and $\kappa$ and the idempotent $j$
constructed in Sections~\ref{sec:lift} and~\ref{sec:corner}, and write
\[
 G=A_\theta,\qquad H_0=H_\theta,\qquad
 \Gamma=(H\times\Gal(K/\Q_p))_\theta.
\]
Here $K$ in $\Gal(K/\Q_p)$ denotes the coefficient field of the
splitting $p$-modular system $(K,\OO,k)$. Let $\lambda\in\Irr(Z)$
be determined by $\theta_Z=\theta(1)\lambda$. As in Section~\ref{sec:corner},
we also denote by $\lambda$ its $\OO$-linear extension
\[
 \lambda:R=\OO Z\longrightarrow\OO,\qquad
 \sum_{z\in Z}a_z z\longmapsto\sum_{z\in Z}a_z\lambda(z),
\]
and set $I_\lambda=\ker\lambda$. This ideal generates two-sided
ideals in $\OO G b$ and $\OO H_0b'$. A subscript $\lambda$ denotes
the corresponding specialization.

\begin{lemma}\label{dgn:lem:fiber-covariance}
For every $a=(h,\gamma)\in\Gamma$, the map $\beta_a$ preserves
$I_\lambda$. If $x\in H_0$ and $y=h^{-1}xh$, then $y\in H_0$ and
\begin{equation}\label{dgn:eq:kappa-fiber-covariance}
 \beta_a(\kappa(x))_\lambda=\kappa(y)_\lambda.
\end{equation}
\end{lemma}

\begin{proof}
Since $\theta^{h\gamma}=\theta$, restriction to $Z$ gives
$\gamma(\lambda(hzh^{-1}))=\lambda(z)$ for all $z\in Z$.
Replacing $z$ by $h^{-1}zh$ and extending linearly, we obtain
\[
 \lambda(\beta_a(r))=\gamma(\lambda(r))\qquad(r\in R).
\]
Thus $\beta_a$ preserves $I_\lambda$ and induces the action of
$\gamma$ on $R/I_\lambda\cong\OO$. Since group and Galois actions
commute, Galois-conjugate characters have the same stabilizer under
the group action. It follows that $h$ normalizes both $G$ and $H_0$;
in particular, $y\in H_0$.

Identify $x,y$ with the elements $(x,1),(y,1)$ of $\Gamma$.
The equality $xa=ay$ and the right crossed identity give
\[
 \varepsilon_a(\sigma(x))\sigma(a)
      =\varepsilon_y(\sigma(a))\sigma(y).
\]
Using $\beta_a(s)=\sigma(a)^{-1}\varepsilon_a(s)\sigma(a)$
for $s\in S$, we therefore obtain the identity over $R$
\begin{equation}\label{dgn:eq:sigma-correction}
 \beta_a(\sigma(x))
   =\sigma(a)^{-1}\varepsilon_y(\sigma(a))\sigma(y).
\end{equation}
The element $y\in H_0$ fixes $\lambda$ and acts trivially on
$\OO$. Hence $\varepsilon_y$ induces the identity on
$S_\lambda\cong M_n(\OO)$, and specialization gives
$\beta_a(\sigma(x))_\lambda=\sigma(y)_\lambda$.
Together with $\beta_a(i)=i$ and $\beta_a(x)=y$, this proves
\eqref{dgn:eq:kappa-fiber-covariance}. The correction factor
$\sigma(a)^{-1}\varepsilon_y(\sigma(a))$ becomes the identity
after specialization; it need not be the identity over $R$.
\end{proof}

\subsection{Equivariance of the character correspondence}

From this point onward, we extend scalars to the coefficient field
$K$ and use the same notation for the resulting algebra elements
and maps. In particular, $S_\lambda\cong M_n(K)$. We use the group-algebra notation $K[X]$ fixed in
Section~\ref{sec:prelim}.

\begin{lemma}\label{dgn:lem:delta-equivariance}
The correspondence
\[
 \Delta_M:\Irr(M\mid\theta)\longrightarrow\Irr(M'\mid\phi)
\]
is $\Gamma$-equivariant.
\end{lemma}

\begin{proof}
Fix $\xi\in\Irr(M\mid\theta)$ and let
$\pi:K[M]\to\End_K(V)$ afford $\xi$. For $x\in M'$, the
element $\kappa(x)_\lambda$ centralizes $S_\lambda\cong M_n(K)$,
because $x$ fixes $\lambda$. On $V_i=\pi(i)V\cong K^n\otimes_K W$,
it therefore acts as $1\otimes\pi'(x)$, where $\pi'$ affords
$\Delta_M(\xi)$. On $(1-\pi(i))V$, it acts as zero. Thus
\begin{equation}\label{dgn:eq:trace-delta}
 \tr\pi(\kappa(x))=n\Delta_M(\xi)(x).
\end{equation}

Let $a=(h,\gamma)\in\Gamma$. The representation
$\pi_a(g)=\pi(hgh^{-1})^\gamma$ affords $\xi^a$. Its
$K$-linear extension satisfies
\[
 \pi_a(\beta_a(u))=\pi(u)^\gamma\qquad(u\in K[M]),
\]
where $\gamma$ is applied entrywise to matrices. For $x\in M'$
and $y=h^{-1}xh$, Lemma~\ref{dgn:lem:fiber-covariance} and
\eqref{dgn:eq:trace-delta} now give
\[
 n\Delta_M(\xi^a)(y)
       =\gamma\bigl(n\Delta_M(\xi)(x)\bigr).
\]
Since $n$ is a positive integer, it is fixed by $\gamma$ and may be
cancelled. We conclude that $\Delta_M(\xi^a)=\Delta_M(\xi)^a$.
This argument establishes equivariance without assuming equality of
the upper mixed stabilizers.
\end{proof}

\subsection{Associated projective representations and synchronization}

\begin{proposition}\label{dgn:prop:common-projectives}
Fix $\xi\in\Irr(M\mid\theta)$ and set $\eta=\Delta_M(\xi)$.
There are projective representations
\[
 \mathcal P:A_\xi\longrightarrow\GL(V),\qquad
 \mathcal P':H_\eta\longrightarrow\GL(W),
\]
associated with $\xi$ and $\eta$, respectively, whose factor sets
agree on $H_\eta\times H_\eta$. For every $c\in C_{A_\xi}(M)$,
the matrices $\mathcal P(c)$ and $\mathcal P'(c)$ are scalar with
the same scalar. Moreover, $H_\xi=H_\eta$, and for every
$a\in(H\times\HH)_\xi$ the comparison functions satisfy
\[
 \mu'_a=\mu_a|_{H_\xi}.
\]
\end{proposition}

\begin{proof}
Since $\theta$ is $M$-invariant, we have $\xi_N=m\theta$ for some
positive integer $m$. Restriction to $N$ therefore gives $A_\xi\leq G$.
If $h\in H$ fixes $\eta$, restriction to $L$ shows that $h$ fixes
$\phi$. By \eqref{dgn:eq:bottom-stabilizers}, it also fixes
$\theta$. Lemma~\ref{dgn:lem:delta-equivariance}, together with
the bijectivity of $\Delta_M$, consequently yields $H_\xi=H_\eta$.
Since $A=MH$ and $M\leq A_\xi$, we also have $A_\xi=MH_\xi$
and a natural isomorphism $A_\xi/M\cong H_\xi/M'$. The factor
sets below are compared through this quotient identification.

Choose a projective representation $\mathcal P$ associated with
$\xi$, with normalized root-of-unity-valued factor set $\alpha$
strictly inflated from $A_\xi/M$. Let
$\pi:K[A_\xi]\to\End_K(V)$ denote its $K$-linear extension.
For $x,y\in A_\xi$, $u\in K[M]x$ and $v\in K[M]y$, we have
\begin{equation}\label{dgn:eq:homogeneous-projective}
 \pi(u)\pi(v)=\alpha(xM,yM)\pi(uv).
\end{equation}
In particular, the restriction of $\pi$ to $K[M]$ is an algebra
representation and factors through the fixed $\lambda$-fiber.

Set $V_i=\pi(i)V$, $J=\pi(j)|_{V_i}$ and $W=JV_i$. Define
\[
 \mathcal P'(x)=\pi(\kappa(x))|_W\qquad(x\in H_\xi).
\]
This is well defined because $\kappa(x)$ centralizes $j$.
Moreover, $\kappa(x)\in K[N]x\subseteq K[M]x$. The
multiplicativity of $\kappa$ and
\eqref{dgn:eq:homogeneous-projective} give
\[
 \mathcal P'(x)\mathcal P'(y)
       =\alpha(xM,yM)\mathcal P'(xy).
\]
Thus the factor set of $\mathcal P'$ is the restriction of that
of $\mathcal P$. On $M'$, the representation $\mathcal P'$ is
the compressed representation affording $\eta$. If
$c\in C_{A_\xi}(M)$, then $c$ centralizes $D\leq M$, so
$c\in H_\xi$. The normalization in Proposition~\ref{dgn:prop:c91-full-integral-mixed} gives
$\sigma(c)=i$ and hence $\kappa(c)=ci$. Since $\mathcal P(c)$
is scalar, its compression $\mathcal P'(c)$ has the same scalar.

We next compare the mixed comparison functions. Fix
$a=(h,\gamma)\in(H\times\HH)_\xi$. The equality
$\xi_N=m\theta$ implies that the image of $a$ in
$H\times\Gal(K/\Q_p)$ belongs to $\Gamma$. We write
$\sigma(a)$ for the realization of this finite image.
The element $h$ normalizes $A_\xi$ and $H_\xi$.
Let $T$ be an intertwining matrix in \eqref{pre:eq:comparison}.
For $u\in K[M]x$ and $y=h^{-1}xh$, linear extension of that
identity gives
\begin{equation}\label{dgn:eq:linear-mixed-intertwiner}
 \pi(u)^\gamma
       =\mu_a(y)T\pi(\beta_a(u))T^{-1}.
\end{equation}
Here $\mu_a$ is constant on $M$-cosets, and $\beta_a$ applies
$\gamma$ to the coefficients.

Since $\beta_a(i)=i$ and $\mu_a|_M=1$, the matrix $T$ maps
$V_i$ onto $V_i^\gamma$. We use the same symbol for this
restriction and set $A_a=\pi(\sigma(a))|_{V_i}$.
The chosen matrix units satisfy $\varepsilon_a(j)=j$, so
$\beta_a(j)=\sigma(a)^{-1}j\sigma(a)$.
Applying \eqref{dgn:eq:linear-mixed-intertwiner} to
$j\in K[N]\subseteq K[M]$, we obtain
\[
 J^\gamma=T A_a^{-1}J A_aT^{-1}.
\]
It follows that
\begin{equation}\label{dgn:eq:corrected-intertwiner}
 U=T A_a^{-1}:V_i\longrightarrow V_i^\gamma,
       \qquad UJU^{-1}=J^\gamma.
\end{equation}
Hence $U$ restricts to an isomorphism $W\to W^\gamma$.

For $x\in H_\xi$ and $y=h^{-1}xh$, apply
\eqref{dgn:eq:linear-mixed-intertwiner} to $\kappa(x)$ and use
\eqref{dgn:eq:kappa-fiber-covariance}. On $V_i$ and $V_i^\gamma$,
this gives
\[
 \pi(\kappa(x))^\gamma
       =\mu_a(y)T\pi(\kappa(y))T^{-1}.
\]
Since $y$ fixes $\lambda$, the element $\kappa(y)_\lambda$
centralizes $S_\lambda$ and thus commutes with $\sigma(a)$ on
the fiber. Also, $\sigma(a)\in K[N]\subseteq K[M]$, and strict
inflation of $\alpha$ gives
$\alpha(M,yM)=\alpha(yM,M)=1$.
Equation~\eqref{dgn:eq:homogeneous-projective} therefore implies
that $\pi(\kappa(y))|_{V_i}$ commutes with $A_a$. We may
consequently replace $T$ by $U$ in the preceding identity on
$V_i$ and $V_i^\gamma$. Restricting further to $W$ and
$W^\gamma$, and writing $x=hyh^{-1}$, we obtain
\[
 \mathcal P'(hyh^{-1})^\gamma
       =\mu_a(y)U\mathcal P'(y)U^{-1},
\]
where $U$ now denotes its restriction to $W$.
Uniqueness of comparison functions yields
$\mu'_a(y)=\mu_a(y)$. Thus the comparison functions agree even
when $J^\gamma\ne J$; their equality follows from the corrected
intertwiner \eqref{dgn:eq:corrected-intertwiner}.

Finally, we realize $\mathcal P'$ over a finite cyclotomic field.
This projective representation is initially defined over the $p$-adic
coefficient field. Its root-of-unity-valued factor set defines a finite central
extension on which it linearizes. Since this finite-group
representation can be realized over a finite cyclotomic field,
a change of basis over $\overline\Q_p$ gives such a realization
without changing the factor set. If
$\widetilde{\mathcal P}'=C^{-1}\mathcal P'C$, the corresponding
intertwiner is $(C^\gamma)^{-1}UC$, so the comparison function
remains $\mu_a$. Thus $\mathcal P$ and $\mathcal P'$ are defined
over $\Qab$, as required in \cite{NSV20}.

For each element of $\HH$, the calculation may be carried out
in a finite extension containing the required coefficients, or
after extending the action to $\Gal(\overline\Q_p/\Q_p)$.
Its restriction to the original coefficient field $K$ determines
the already chosen $\sigma(a)$. Enlarging the coefficient field
therefore leaves the chosen realization $\sigma$ unchanged.
\end{proof}

\subsection{Identification of the tensor and corner correspondences}

\begin{lemma}\label{dgn:lem:same-tensor-model}
With the notation of Proposition~\ref{dgn:prop:common-projectives},
let $M\leq U\leq A_\xi$. The tensor correspondence defined by
$\mathcal P,\mathcal P'$ agrees on $\Irr(U\mid\xi)$ with the
corner correspondence defined by $\kappa$.
\end{lemma}

\begin{proof}
Let $\mathcal Q\otimes\mathcal P_U$ afford
$\chi\in\Irr(U\mid\xi)$, where $\mathcal Q$ is an irreducible
projective representation of $U/M$ whose factor set inflates to
the inverse of that of $\mathcal P_U$. In this tensor product,
$\mathcal Q$ is inflated to $U$.
Let $E$ be the representation space of $\mathcal Q$, so that
$\chi$ is afforded on $E\otimes_K V$.
An element $u\in K[N]\subseteq K[M]$ acts as
$1_E\otimes\pi(u)$. Applying the idempotent $i$ therefore gives
$E\otimes_K V_i$, and applying $j$ next gives
\[
 (1_E\otimes J)(E\otimes_K V_i)=E\otimes_K W.
\]
For $x\in U\cap H$, the fact that $\kappa(x)\in K[N]x$
allows us to write
\[
 \kappa(x)=\sum_{n\in N}c_nnx\qquad(c_n\in K).
\]
Since $N\leq M$, all the elements $nx$ have the same image $xM$
in $U/M$. Hence $\kappa(x)$ acts on $E\otimes_K V$ as
$\mathcal Q(xM)\otimes\pi(\kappa(x))$. Its restriction to
$E\otimes_K W$ is therefore
\[
 \mathcal Q(xM)\otimes\mathcal P'(x).
\]
The tensor and corner constructions thus give the same
representation on $E\otimes_K W$. A subsequent
change of basis $C$ realizing $\mathcal P'$ over a cyclotomic
field changes this tensor representation by $1_E\otimes C$,
and hence leaves the character correspondence unchanged.
\end{proof}

\section{Block compatibility}\label{sec:blocks}

We prove that the character triple correspondence constructed in
Section~\ref{sec:mixed} satisfies the block condition. We first
identify its block correspondence on the test groups defined below, using
central quotients and the normalized residual magic construction.
The normalized trace criterion of Navarro and Sp\"ath then gives
block compatibility for every intermediate group.

\subsection{The fixed correspondence and the test groups}
\label{dgn:subsec:central91-blocks}

Work over the sufficiently large splitting $p$-modular system
$(K,\OO,k)$ fixed above. We use the configuration
\[
 \begin{gathered}
 N\leq M\trianglelefteq G,\qquad M=ND,\qquad
 Z=N\cap D\leq Z(M),\\
 H=N_G(D),\qquad L=N_N(D),\qquad M'=N_M(D).
 \end{gathered}
\]
Throughout this section, we write $G=A_\xi$ and $H=N_G(D)$,
where $\xi\in\Irr_0(M\mid\theta)$ is fixed. Thus
$N\trianglelefteq G$, $G=MH$, and $G$ fixes both $\theta$ and
$\xi$. Let $\phi$ and $\eta\in\Irr_0(M'\mid\phi)$ be their
local correspondents. The character $\eta$ is $H$-invariant.
As before, $b,b'$ are the base blocks, $i=bb'$, and $R=\OO Z$.
The blocks $\bl(\xi)$ and $\bl(\eta)$ both have defect group $D$.

\begin{definition}\label{dgn:def:test-group}
Let $M\trianglelefteq G$, and let $D\leq M$ be a $p$-subgroup.
A subgroup $J\leq G$ is called a \emph{test group for $(M,D)$} if
\[
 J=M\langle h\rangle
\]
for some $p'$-element $h\in C_G(D)$.
\end{definition}

For such a group $J$, the quotient $J/M$ is a cyclic $p'$-group.
Moreover, every subgroup $J_0$ with $M\leq J_0\leq J$ is again a
test group, since $J_0=M\langle h^r\rangle$ for some integer $r$.
In our configuration, $Z\leq D\cap Z(M)$, so $Z\leq Z(J)$.
When applying the normalized trace criterion in
Theorem~\ref{dgn:thm:central91-all-blocks}, we choose $h$ with the
additional property $D\in\Syl_p(C_M(h))$.

Proposition~\ref{dgn:prop:c91-graded-corner},
Lemma~\ref{dgn:lem:c91-full-residual},
Proposition~\ref{dgn:prop:common-projectives} and
Lemma~\ref{dgn:lem:same-tensor-model} provide the following data.
\begin{enumerate}
\item A fixed integral magic construction $\sigma$ determines graded
corner isomorphisms $\kappa_U$ for $M\leq U\leq G$, where
$U'=U\cap H$. These maps use the same matrix units in
$S\cong M_n(R)$ and satisfy
\[
 \kappa_U(yb')=y\sigma(y)^{-1}\qquad(y\in U').
\]
The associated matrix corners induce Morita correspondences between
$\OO Ub$ and $\OO U'b'$. Central idempotents $B$ of $\OO Ub$ and
$c$ of $\OO U'b'$ correspond precisely when $Bi=\kappa_U(c)$.
\item If $Z\leq Z(U)$, specialization at $z\mapsto1$, followed by
reduction modulo $J(\OO)$, gives the corner isomorphism for $U/Z$
determined by the residual magic construction. Its multiplicity
algebra $\bar S\cong M_n(k)$ has Brauer quotient $k$ at $P=D/Z$.
Let $\rho:P\to\bar S^\times$ be the unique homomorphism
implementing the $P$-action. The fusion homomorphism and the normalized
implementers satisfy
\[
 \Phi(\widehat T_y)=1\qquad(y\in H).
\]
Here $y$ is regarded as a pure group element. It fixes $\theta$
and acts trivially on $\OO$, so its residual field action is trivial
and $\widehat T_y=\bar\sigma(y)$ is $k$-linear. On units
centralizing $\rho(P)$, the map $\Phi$ agrees with
$\operatorname{Br}^{\bar S}_P$.
\item The fixed projective representations $\mathcal P,\mathcal P'$
associated with $\xi,\eta$ have a common root-of-unity-valued
factor set $\alpha$, inflated from $G/M\cong H/M'$, and equal
scalars on $C_G(M)$. Moreover, $\mathcal P'$ is obtained from
$\mathcal P$ by compression through $\kappa$. For every projective
representation $\mathcal Q$ with factor set $\alpha^{-1}$, the
corner compression of $\mathcal Q\otimes\mathcal P$ is
$\mathcal Q\otimes\mathcal P'$.
\end{enumerate}
The final assertion in (ii) follows from
\cite[Theorem~4.2]{Fu26}. This normalization is essential: the
argument requires the residual corner isomorphism induced by the
fixed construction. It does not require an identity
$\operatorname{Br}_D(\sigma(y))=1$ on
$S/J(\OO)S\cong M_n(kZ)$.

\subsection{Central quotients and Brauer maps}

\begin{lemma}\label{dgn:lem:central91-brauer-quotient}
Let $V$ be a finite group, let $Z\leq Z(V)$ be a $p$-subgroup,
and let $D$ be a $p$-subgroup of $V$ containing $Z$. Set
$\bar V=V/Z$, $P=D/Z$, and
let $q:kV\to k\bar V$ be the natural epimorphism.
For every $a\in(kV)^D$,
\begin{equation}\label{dgn:eq:central91-brauer-quotient}
 q\bigl(\operatorname{Br}_D^V(a)\bigr)
     =\operatorname{Br}_P^{\bar V}\bigl(q(a)\bigr).
\end{equation}
On the left-hand side, $q$ is induced by the natural map
$C_V(D)\to C_{\bar V}(P)$, which need not be surjective.
Moreover, $q$ induces a bijection between the central idempotents
of $kV$ and those of $k\bar V$, preserving primitivity.
\end{lemma}

\begin{proof}
Write $a=\sum_{v\in V}a_vv$. Since $a$ is $D$-invariant, its
coefficients are constant on $D$-conjugacy orbits. If $vZ\in\bar V$
is not fixed by $P$, its coefficient is zero on both sides of
\eqref{dgn:eq:central91-brauer-quotient}.  If
$vZ\in C_{\bar V}(P)$, the entire coset $vZ$ is $D$-stable.
If one element of this coset centralizes $D$, then every element
does, since $Z\leq Z(V)$. Both coefficients are then
$\sum_{z\in Z}a_{vz}$.  Otherwise the coset is a union of nontrivial
$D$-orbits, each of length a positive power of $p$.
The sum of the coefficients on each such orbit is zero in $k$, and
both coefficients are again zero. This proves
\eqref{dgn:eq:central91-brauer-quotient}, without requiring
surjectivity of the map on centralizers.

Let $I$ be the augmentation ideal of $kZ$. Then $I$ is nilpotent,
$\ker q=I(kV)$, and $I$ is contained in the center of $kV$.
Idempotent lifting modulo a nilpotent ideal gives a complete set of
orthogonal lifts of any complete set of orthogonal idempotents in
the quotient. Suppose the quotient idempotents are central, and let
$e,f$ be lifts of distinct ones. Then
\[
 e(kV)f\subseteq I(kV),\qquad e(kV)f=I\,e(kV)f.
\]
The second equality follows from the centrality of $I$. Iteration
and nilpotence give $e(kV)f=0$ and, similarly, $f(kV)e=0$.
Thus the lifts are central. If central idempotents $e,f$ have the
same image, then
$e(1-f),f(1-e)\in\ker q$.  Each is an idempotent in a nilpotent
ideal, and hence is zero. Thus $e=f$, proving the asserted
bijection. Primitivity is preserved as a consequence.
\end{proof}

\subsection{Block correspondence on the test groups}

\begin{lemma}\label{dgn:lem:central91-test-blocks}
Let $h\in H$ be a $p'$-element with $D\leq C_M(h)$, and set
$J=M\langle h\rangle$. For every $M\leq J_0\leq J$, the blocks
$B,c$ paired by $\kappa_{J_0}$ satisfy
\begin{equation}\label{dgn:eq:central91-test-blocks}
 \operatorname{Br}_D^{J_0}(B)=c,\qquad c^{J_0}=B.
\end{equation}
Here $c$ is a block of $J_0'=N_{J_0}(D)=J_0\cap H$, and the
idempotents in the first identity are understood after reduction
modulo $J(\OO)$.
\end{lemma}

\begin{proof}
Since $Z\leq Z(M)$, $Z\leq D$ and $h\in C_G(D)$, we have
$Z\leq Z(J)$ and hence $Z\leq Z(J_0)$.  Also $Z\leq D$, so
\[
 N_{J_0/Z}(D/Z)=N_{J_0}(D)/Z=J_0'/Z.
\]
We use bars for the modular objects obtained after passage to the
central quotient by $Z$, and set $P=D/Z$.
Lemma~\ref{dgn:lem:central91-brauer-quotient} shows that
$\bar B,\bar c$ are block idempotents in the corresponding quotient
groups.  Reduction compatibility of the integral corner isomorphism
gives
\begin{equation}\label{dgn:eq:central91-reduced-corner}
 \bar B\,\bar i=\bar\kappa_{J_0}(\bar c).
\end{equation}

We first verify the Brauer normalization in the quotient. Let
$yZ\in C_{J_0'/Z}(P)$. Conjugation by $y$ acts trivially on $P$.
Since $y\in G$ fixes $\theta_0$, its residual field action is
trivial, and $\bar\sigma(y)$ implements its action on $\bar S$.
The uniqueness of the implementing homomorphism $\rho$ gives
\[
 \bar\sigma(y)\rho(u)\bar\sigma(y)^{-1}=\rho(u)
                                             \qquad(u\in P).
\]
Consequently, $\bar\sigma(y)\in(\bar S^P)^\times$. The
normalization supplied by the fusion homomorphism therefore gives
\begin{equation}\label{dgn:eq:central91-quotient-norm}
 \operatorname{Br}^{\bar S}_P(\bar\sigma(y))
       =\Phi(\widehat T_y)=1.
\end{equation}
The $P$-equivariant inclusion
$\bar S\hookrightarrow\bar i\,k(N/Z)\bar i$ commutes with
relative trace maps and therefore induces a homomorphism on
Brauer quotients. Under the scalar identification $\bar S(P)=k$,
this homomorphism sends $t\in k$ to $t\bar b'$, since
$1_{\bar S}=\bar i$ and the defect-zero quotient configuration
gives $\operatorname{Br}_P(\bar i)=\bar b'$.
Thus \eqref{dgn:eq:central91-quotient-norm} implies
$\operatorname{Br}_P(\bar\sigma(y))=\bar b'$ in the group algebra.
The element $y$ itself need not centralize $D$.

We now apply the calculation in
\cite[proof of Corollary~11.2]{Lad11} to $k(J_0/Z)$.
Since $P\trianglelefteq J_0'/Z$, the support property of block
idempotents in the presence of a normal $p$-subgroup gives
\[
 \bar c=\sum_{\bar y\in C_{J_0'/Z}(P)}c_{\bar y}\bar y.
\]
Moreover, $\operatorname{Br}_P(\bar b)=\bar b'$.
Together with \eqref{dgn:eq:central91-reduced-corner} and
\eqref{dgn:eq:central91-quotient-norm}, this yields
\begin{align*}
 \operatorname{Br}_P(\bar B)
 &=\operatorname{Br}_P(\bar B)\bar b'
   =\operatorname{Br}_P(\bar B\,\bar i)\\
 &=\operatorname{Br}_P(\bar\kappa_{J_0}(\bar c))\\
 &=\sum_{\bar y\in C_{J_0'/Z}(P)}c_{\bar y}\bar y\,
       \operatorname{Br}_P(\bar\sigma(y)^{-1})
   =\bar c.
\end{align*}

Let $q:kJ_0'\to k(J_0'/Z)$.  By
\eqref{dgn:eq:central91-brauer-quotient},
\[
 q(\operatorname{Br}_D(B))
       =\operatorname{Br}_P(\bar B)=\bar c=q(c).
\]
The element $\operatorname{Br}_D(B)$ is a $J_0'$-invariant
idempotent, so it and $c$ are central idempotents of $kJ_0'$.
The uniqueness of central idempotent lifts modulo the nilpotent
kernel of $q$ gives $\operatorname{Br}_D(B)=c$. Thus this equality
identifies the blocks paired by the fixed map $\kappa_{J_0}$.

The blocks under consideration cover $b$, and hence cover the unique
block $\bl(\xi)$ of $M$ above $b$.
Since $J_0/M$ is a $p'$-group, the intersection property for defect
groups of covering blocks shows that every block of $J_0$ covering
$\bl(\xi)$ has defect group $D$. The same argument applies
locally. Equivalently, \cite[Lemma~3.5(c)]{NS14}, applied locally,
gives a defect group $Q$ with $Q\cap M'=D$. As $J_0'/M'$ is a
$p'$-group, this implies $Q=D$. Now $J_0'=N_{J_0}(D)$, so Brauer's
first main theorem yields $c^{J_0}=B$.
\end{proof}

\subsection{Trace identities and all intermediate groups}

\begin{theorem}\label{dgn:thm:central91-all-blocks}
The character triple correspondence determined by the pair
$\mathcal P,\mathcal P'$ above satisfies
\begin{equation}\label{dgn:eq:central91-all-blocks}
 \bl(\tau_U(\chi))^U=\bl(\chi)
 \qquad(M\leq U\leq G,\ \chi\in\Irr(U\mid\xi)).
\end{equation}
\end{theorem}

\begin{proof}
The common factor set and the equality of central scalars give a
central character triple isomorphism. Both base blocks have defect
group $D$, and $N_M(D)=M'$. Hence the group and defect-group
hypotheses of \cite[Theorem~4.4]{NS14} are satisfied. It remains
to verify the normalized trace condition in that theorem.

Let $h\in H$ be a $p'$-element satisfying
$D\in\operatorname{Syl}_p(C_M(h))$.  In particular,
$h\in C_G(D)$, so Lemma~\ref{dgn:lem:central91-test-blocks} applies
to $J=M\langle h\rangle$.
Since $J/M$ is cyclic, there is a one-dimensional projective
representation $\mathcal Q$ of $J/M$ over a finite cyclotomic
extension of the coefficient field, with factor set
$\alpha^{-1}|_{J/M}$.
Let $\OO'$ be the valuation ring of this enlarged coefficient field.
We may choose $\mathcal Q$ to take values in a finite group of roots
of unity: the recursive construction along a cyclic generator
requires only roots of roots of unity. In particular,
$\mathcal Q(hM)$ is a unit in $\OO'$.
We extend scalars in the fixed constructions $\sigma,\kappa$;
the residual Brauer identities remain valid after the corresponding
finite extension of the residue field.

Define the ordinary characters
\[
 \psi=\operatorname{tr}(\mathcal Q\otimes\mathcal P_J),\qquad
 \psi'=\operatorname{tr}(\mathcal Q\otimes\mathcal P'_{J'}),
 \qquad J'=J\cap H.
\]
These characters extend $\xi$ and $\eta$, respectively, and are
therefore irreducible. The compatibility of corner compression with
tensor factors shows that they correspond under $\kappa_J$.
For every $M\leq J_0\leq J$, their restrictions again extend
$\xi,\eta$ and correspond under $\kappa_{J_0}$.
Lemma~\ref{dgn:lem:central91-test-blocks} gives
\[
 \bl(\psi'_{J_0\cap H})^{J_0}=\bl(\psi_{J_0}).
\]
For $a\in\OO'$, write $a^*$ for its image under the residue map
\[
 \OO'\longrightarrow\OO'/J(\OO').
\]
We apply \cite[Lemma~4.2(b)]{NS14} to the base groups $M,M'$, the
upper groups $J,J'$, the characters $\xi,\eta$, and their
extensions $\psi,\psi'$. Since both base characters have height
zero, we obtain
\[
 \left(\frac{|M|_{p'}\psi(h)}{\xi(1)_{p'}}\right)^*
  =\left(\frac{|M'|_{p'}\psi'(h)}{\eta(1)_{p'}}\right)^*.
\]
Cancelling the common nonzero factor $\mathcal Q(hM)^*$ gives
\begin{equation}\label{dgn:eq:central91-trace-test}
 \left(
  \frac{|M|_{p'}\operatorname{tr}\mathcal P(h)}{\xi(1)_{p'}}
 \right)^*
 =
 \left(
  \frac{|M'|_{p'}\operatorname{tr}\mathcal P'(h)}{\eta(1)_{p'}}
 \right)^*.
\end{equation}
The integrality of these normalized traces in $\OO'$ also follows
from \cite[Lemma~4.3]{NS14}. In the preceding calculation, we
have divided only by the root of unity $\mathcal Q(hM)$, which
is a unit in $\OO'$.

The identity holds for every test element required by
\cite[Theorem~4.4]{NS14}, which therefore gives
\eqref{dgn:eq:central91-all-blocks}. The conclusion applies to the
fixed pair $\mathcal P,\mathcal P'$ whose mixed comparison
functions were synchronized in Section~\ref{sec:mixed}.
\end{proof}

The proof uses the centrality of $Z$ only in the test groups
$J=M\langle h\rangle$ and their subgroups, and therefore does not
require $Z\leq Z(G)$. In each quotient argument, the subgroup being
factored out is a central $p$-subgroup, so the required block
bijection is available. Moreover, the integral corner equivalence
includes every central-character fiber; no character lying over a
nontrivial character of $Z$ is required to descend to $M/Z$.

\section{Proof of the main theorem and consequences}\label{sec:completion}

We first extend the correspondence on a fixed character fiber to the
full Galois orbit. We then verify the block $\HH$-triple relation in
Theorem~A and derive two consequences.

\subsection{Gluing the Galois fibers}
\label{dgn:sec:central91-gluing}

We use the notation of Theorem~A. Thus $N\leq M$ are normal
subgroups of $A$, the quotient $M/N$ is a $p$-group, and
$b\in\Bl(N)$ has defect group $Z\leq Z(M)$. The character
$\theta\in\Irr(b)$ is $M$-invariant and satisfies
\[
 \theta^a\in\theta^{\HH}\qquad\text{for every }a\in A.
\]
Let $B$ be the unique block of $M$ covering $b$, choose a defect
group $D$ of $B$, and set
\[
 H=N_A(D),\qquad L=N_N(D),\qquad M'=N_M(D)=LD.
\]
Let $\phi$ be the generalized DGN correspondent of $\theta$, and
write $b_0=\bl(\phi)$. The unique block $B_0$ of $M'$ covering
$b_0$ is the Brauer correspondent of $B$.

Lemma~\ref{dgn:lem:delta-equivariance},
Proposition~\ref{dgn:prop:common-projectives},
Lemma~\ref{dgn:lem:central91-height-fibers} and
Theorem~\ref{dgn:thm:central91-all-blocks} provide a
mixed-equivariant correspondence of height-zero characters on the
$\theta$-fiber. All the required representation and block conditions
are realized by the same corner construction.
Set
\[
 \mathcal T=H\times\HH,\qquad
 \mathcal T_\theta=(H\times\HH)_\theta.
\]
The calculations for $\mathcal T_\theta$ are carried out at finite
coefficient levels using the finite mixed stabilizer $\GG$ from
Section~\ref{sec:lift}. By \eqref{dgn:eq:bottom-stabilizers},
we have $\mathcal T_\theta=\mathcal T_\phi$.

\begin{proposition}\label{dgn:prop:central91-gluing}
The correspondence on the $\theta$-fiber extends uniquely to a
$\mathcal T$-equivariant bijection
\[
 \Delta_D:\Irr_0(M\mid\theta^{\HH})
       \longrightarrow\Irr_0(M'\mid\phi^{\HH}).
\]
For every $\gamma\in\HH$, it maps the fiber above $\theta^\gamma$
to the fiber above $\phi^\gamma$, and sends each character to a
character in the Brauer correspondent of its block.
\end{proposition}

\begin{proof}
Since $\theta^h\in\theta^{\HH}$ for every $h\in H$, the
$\mathcal T$-orbit of $\theta$ is exactly $\theta^{\HH}$.
Naturality of the generalized DGN correspondence gives the analogous
statement for $\phi$. By
Lemma~\ref{dgn:lem:central91-height-fibers} and Galois transport,
each character in $\Irr_0(M\mid\theta^{\HH})$ restricts
irreducibly to $N$.
Thus the fibers above distinct elements of $\theta^{\HH}$ are
disjoint. Likewise, every character in
$\Irr_0(M'\mid\phi^{\HH})$ restricts irreducibly to $L$, so
the fibers above distinct elements of $\phi^{\HH}$ are disjoint.

For $\xi_0\in\Irr_0(M\mid\theta)$ and $t\in\mathcal T$, define
\[
 \Delta_D(\xi_0^t)=\Delta_\theta(\xi_0)^t.
\]
Suppose that $\xi_0^t=\xi_1^{t_1}$, where
$\xi_1\in\Irr_0(M\mid\theta)$ and $t_1\in\mathcal T$.
Restriction to $N$ gives $\theta^t=\theta^{t_1}$. Hence
$u=tt_1^{-1}\in\mathcal T_\theta$ and $\xi_1=\xi_0^u$.
Mixed equivariance on the $\theta$-fiber yields
\[
 \Delta_\theta(\xi_1)^{t_1}
 =\Delta_\theta(\xi_0^u)^{t_1}
 =\Delta_\theta(\xi_0)^{ut_1}
 =\Delta_\theta(\xi_0)^t.
\]
Thus $\Delta_D$ is well defined and $\mathcal T$-equivariant.
Since $\mathcal T_\theta=\mathcal T_\phi$, the fibers on both
sides are indexed by the same coset space
$\mathcal T_\theta\backslash\mathcal T$. The bijections on these
fibers therefore combine to give a bijection, and equivariance makes
this extension unique. Taking $t=(1,\gamma)$ gives the asserted
property of the prescribed Galois fibers. Block compatibility follows
from the corresponding property on the $\theta$-fiber and the
naturality of Brauer correspondence under group and Galois actions.

This argument uses equivariance under the full mixed stabilizer
$\mathcal T_\theta$. Separate equivariance under $H_\theta$ and
$\HH_\theta$ would not suffice, since $tt_1^{-1}$ need not belong
to $H_\theta\times\HH_\theta$.
\end{proof}

\subsection{Stabilizers and the overgroup relation}

\begin{theorem}\label{dgn:thm:central91-final}
Theorem~A holds. In particular, if $\eta=\Delta_D(\xi)$, then
\[
 (A_{\xi^{\HH}},M,\xi)_{\HH}
       \geb(H_{\eta^{\HH}},M',\eta)_{\HH}.
\]
\end{theorem}

\begin{proof}
First suppose that $\xi\in\Irr_0(M\mid\theta)$, and set
\[
 X=A_{\xi^{\HH}},\qquad Y=H\cap X.
\]
The equivariance and injectivity of $\Delta_D$ give
\[
 Y=H_{\eta^{\HH}},\qquad
 (Y\times\HH)_\xi=(Y\times\HH)_\eta.
\]
By Lemma~\ref{dgn:lem:central91-frattini}, we have $A=MH$.
Since inner automorphisms of $M$ fix every ordinary character of
$M$, we have $M\leq X$. If $x=mh\in X$, with $m\in M$ and
$h\in H$, then $h\in H\cap X$. Consequently,
\[
 X=MY,\qquad M\cap Y=M\cap H=M',\qquad Y=N_X(D).
\]
Moreover,
\[
 C_X(M)\leq C_X(D)\leq Y,\qquad N_M(D)=M'.
\]
The group $D$ is a common defect group of $\bl(\xi)$ and
$\bl(\eta)$. The required condition on its centralizer is
\[
 C_{X_\xi}(D)\leq Y_\xi=Y_\eta.
\]

We next verify the representation conditions in the block relation.
Restriction to $N$ gives $X_\xi=A_\xi\leq A_\theta$, and we have
$Y_\eta=Y_\xi=H\cap A_\xi$. Thus the construction on the
$\theta$-fiber is available on the full inertia groups occurring in
the desired relation. Let $\mathcal P,\mathcal P'$ be the associated
projective representations provided by
Proposition~\ref{dgn:prop:common-projectives}.
Every $a=(h,\gamma)\in(Y\times\HH)_\xi$ belongs to
$\mathcal T_\theta$, so the synchronization formula gives
\[
 \mu'_a=\mu_a|_{Y_\xi}.
\]
The factor sets of this pair agree on $Y_\xi\times Y_\xi$, and
the pair has equal central scalars. For every $M\leq J\leq X_\xi$,
its associated correspondence satisfies
\[
 \bl(\tau_J(\chi))^J=\bl(\chi)
                         \qquad(\chi\in\Irr(J\mid\xi)),
\]
by Lemma~\ref{dgn:lem:same-tensor-model} and
Theorem~\ref{dgn:thm:central91-all-blocks}.
Hence all conditions of the block $\HH$-triple relation are
satisfied by the same pair $\mathcal P,\mathcal P'$.

If $\xi$ lies above another element of $\theta^{\HH}$, choose
$t\in\mathcal T$ transporting it to the $\theta$-fiber.
Transport the two projective representations and the entire family
of intermediate-group correspondences simultaneously. Group
isomorphisms and valuation-preserving Galois automorphisms preserve
the factor-set, central-scalar and comparison-function identities,
as well as Brauer correspondence. Since $H$ normalizes $D$ and
Lemma~\ref{dgn:lem:central91-frattini} preserves $D$ as a defect
group under Galois transport, the transported pair realizes the
required relation on every fiber. This proves Theorem~A.
\end{proof}

The proof combines the generalized DGN correspondence, the Dade
multiplicity algebra, the finite-field fusion homomorphism and the
block trace criterion. The normalization under the full mixed
stabilizer makes the fiber covariance and the central-test
calculation compatible within the same corner construction.
Together with Clifford theory and centralization arguments, this
correspondence can be used in reduction arguments beyond the
normal $p$-section considered here.

\subsection{The defect-zero case}

\begin{corollary}\label{cor:defect-zero}
In Theorem~A, assume that $Z=1$. Then $N\cap D=1$,
$L=C_N(D)$ and $M'=D\times C_N(D)$. The correspondence
$\Delta_D$ has all the equivariance, height-zero and block
$\HH$-triple properties stated in Theorem~A.
\end{corollary}

\begin{proof}
Only the group identities require verification. For $x\in N_N(D)$,
we have $[x,D]\leq N\cap D=1$, so $N_N(D)=C_N(D)$. The identity
$N_M(D)=DN_N(D)$ therefore gives
$M'=D\times C_N(D)$. The remaining assertions follow from
Theorem~A.
\end{proof}

When $Z=1$, the block $b$ has defect zero, giving the corresponding
configuration in the magic-representation approach. Here the
conclusion is realized by one pair of ordinary projective
representations with synchronized mixed comparison functions and
compatible block correspondences on every intermediate group.

\subsection{A form for the Alperin--McKay--Navarro reduction}

The following consequence is stated without fixing an irreducible
character of $N$ and is suited to a normal section in an
Alperin--McKay--Navarro reduction. For the surrounding reduction
frameworks, see \cite{Spa13,NS14} for the block conditions and
\cite{Nav04,NSV20} for the Galois conditions.

\begin{corollary}\label{int:cdgn}
Let $N\leq V$ be normal subgroups of a finite group $B$, and let
$D\leq V$ be a $p$-subgroup. Suppose that
\[
 V=ND,\qquad Z=N\cap D\leq Z(V).
\]
Then there is an $N_B(D)\times\HH$-equivariant bijection
\[
 \Delta:\Irr_0(V\mid D)\longrightarrow\Irr_0(N_V(D)\mid D)
\]
that sends each character to a character in the Brauer correspondent
of its block and satisfies
\[
 (B_{\xi^{\HH}},V,\xi)_{\HH}\geb
 \bigl((N_B(D))_{(\xi')^{\HH}},N_V(D),\xi'\bigr)_{\HH}
 \quad\text{for }\xi'=\Delta(\xi).
\]
\end{corollary}

\begin{proof}
Set $C=N_B(D)$ and $L=N_N(D)$, and fix
$\xi\in\Irr_0(V\mid D)$.
By \cite[Proposition~2.5]{NS14}, we may choose a constituent
$\theta\in\Irr_0(N\mid Z)$ of $\xi_N$ such that the Clifford
correspondent of $\xi$ over $\theta$ has defect group $D$.
Set $I=V_\theta$, and let $\psi\in\Irr(I\mid\theta)$ be this
Clifford correspondent. Then
\[
 \Ind_I^V(\psi)=\xi,\qquad \psi_N=e\theta
\]
for some positive integer $e$.
Choose a projective representation $\mathcal P_\theta$ of $I$
associated with $\theta$, with factor set inflated from $I/N$.
Projective Clifford theory expresses a representation affording
$\psi$ as $\mathcal Q\otimes\mathcal P_\theta$, where
$\mathcal Q$ is inflated from an irreducible projective
representation of $I/N$ with the inverse factor set.
Thus $e=\dim\mathcal Q$, the dimension of its representation
space. This is the projective degree appearing in the Clifford
degree formula
\[
 \xi(1)=[V:I]\,e\,\theta(1).
\]
Since $\theta$ and $\xi$ have height zero in blocks with defect
groups $Z$ and $D$, respectively, we have
\[
 \theta(1)_p=|N:Z|_p=|V:D|_p=\xi(1)_p.
\]
The quotient $V/N$ is a $p$-group, so $[V:I]$ is a power of $p$.
To see that $e$ is also a power of $p$, regard $\mathcal Q$ as a
complex projective representation of $I/N$.
The cohomology group $H^2(I/N,\C^\times)$ is annihilated by
$|I/N|$, so multiplying the matrices of $\mathcal Q$ by suitable
scalars makes its factor set take values in the $|I/N|$-th roots
of unity. The resulting projective representation lifts to an
irreducible ordinary representation of the corresponding finite
central extension of $I/N$, which is a $p$-group.
Its degree is still $e$, and hence $e$ is a power of $p$.
Taking $p$-parts in the degree formula therefore gives
$[V:I]e=1$. Hence $I=V$, $e=1$ and $\xi_N=\theta$.
In particular, $\theta$ is $V$-invariant.
Applying the same argument to $L\trianglelefteq N_V(D)=LD$
shows that every character in $\Irr_0(N_V(D)\mid D)$ restricts
irreducibly to $L$.

For the irreducible characters of $N$ and $L$ obtained by these
restrictions, the generalized DGN correspondence
$\theta\mapsto\phi$ preserves central characters and the specified
central defect group, and commutes with group and Galois actions.
The blocks and character fibers that occur on both sides are
identified in \cite[Section~5]{NS14}; they can also be described
using the $\OO Z$-matrix fibers in Section~\ref{sec:central}.
Hence the correspondence between these two sets of irreducible
characters is $C\times\HH$-equivariant.

Choose one representative $\theta$ from each $C\times\HH$-orbit
of the characters of $N$ obtained above, and write
$B_0=B_{\theta^{\HH}}$.
We apply Theorem~A with ambient group $B_0$. Indeed,
$N,V\nrm B_0$, the character $\theta$ is $V$-invariant, and
$\bl(\theta)$ has defect group $Z\leq Z(V)$. The unique block
of $V$ covering $\bl(\theta)$ contains the chosen $\xi$, and
therefore has defect group $D$. The theorem gives a correspondence
on the $\theta$-fiber that is equivariant under the entire mixed
stabilizer $(C\times\HH)_\theta$.

Extend these correspondences by setting
\[
 \Delta(\xi^t)=\Delta(\xi)^t\qquad(t\in C\times\HH)
\]
for $\xi$ in a chosen fiber. If a character has two such
expressions, restriction to $N$ shows that the transporting
elements differ by an element of $(C\times\HH)_\theta$.
Mixed equivariance on the chosen fiber therefore makes the
definition independent of the expression. Since the correspondence
between the occurring characters of $N$ and $L$, as well as the
correspondence on each fixed fiber, is bijective, $\Delta$ is a
bijection. Its block property
follows from the block property on each fiber and the naturality
of Brauer correspondence.

Finally, $\xi_N=\theta$ implies $B_{\xi^{\HH}}\leq B_0$.
The block $\HH$-triple relation obtained in $B_0$ is therefore
defined on the full overgroup required for $\xi$. Transporting
the pair of projective representations realizing this relation
simultaneously transports the common factor sets, central scalars,
comparison functions and intermediate-group block correspondences.
This gives the stated relation for every $\xi$.
\end{proof}

\Needspace{10\baselineskip}

\end{document}